\documentclass{amsart}

\newtheorem{theorem}{Theorem}[section]

\newtheorem{lemma}[theorem]{Lemma}

\usepackage[abbrev]{amsrefs}
\usepackage{hyperref}

\usepackage{amsmath}
\usepackage{amssymb}

\usepackage{enumerate}

\newtheorem{proposition}[theorem]{Proposition}

\theoremstyle{definition}
\newtheorem{definition}[theorem]{Definition}

\newtheorem*{theorem-A}{Theorem A}
\newtheorem*{theorem-B}{Theorem B}

\theoremstyle{remark}
\newtheorem{remark}[theorem]{Remark}

\numberwithin{equation}{section}

\newcommand{\R}{\mathbb{R}}

\newcommand{\N}{\mathbb{N}}

\newcommand{\rn}{\R^n}

\newcommand{\M}{\mathfrak{M}}

\newcommand{\intdif}[1]{\, \mathrm{d}{#1}}
\newcommand{\dt}{\intdif{t}}
\newcommand{\ds}{\intdif{s}}
\newcommand{\dtau}{\intdif{\tau}}

\newcommand{\dy}{\intdif{y}}

\newcommand{\pone}{q}
\newcommand{\ptwo }{r}

\DeclareMathOperator{\spt}{supp}

\newcommand{\RR}{\mathcal{R}}
\newcommand{\cS}{\mathcal{S}}

\begin{document}

\title{Potential trace inequalities via a Calder\'on-type theorem}

\begin{abstract}
In this paper we develop a general theoretical tool for the establishment of the boundedness of notoriously difficult operators (such as potentials) on certain specific types of rearrangement-invariant function spaces from analogous properties of operators that are easier to handle (such as fractional maximal operators).  A principal example of the new results one obtains by our analysis is the following inequality, which generalizes a result of Korobkov and Kristensen (who had treated the case $\mu=\mathcal{L}^n$, the Lebesgue measure on $\mathbb{R}^n$):  There exists a constant $C>0$ such that
\begin{align*}
\int_{\mathbb{R}^n} |I_\alpha^\mu f|^p \intdif{\nu} \leq C \|f\|_{L^{p,1}(\mathbb{R}^n,\mu)}^p
\end{align*}
for all $f$ in the Lorentz space $L^{p,1}(\mathbb{R}^n,\mu)$, where $\mu, \nu$ are Radon measures such that
\begin{equation*}
\sup_{Q} \frac{\mu(Q)}{l(Q)^{d}} < \infty \quad \text{and} \quad        \sup_{\mu(Q)>0} \frac{\nu(Q)}{\quad\mu(Q)^{1-\frac{\alpha p}{d}}} < \infty,
\end{equation*}
 and $I_\alpha^\mu$ is the Riesz potential defined with respect to $\mu$ of order $\alpha \in (0,d)$.  More broadly, we obtain inequalities in this spirit in the context of rearrangement-invariant spaces through a result of independent interest, an extension of an interpolation theorem of Calder\'on where the target space in one endpoint is a space of bounded functions.        
 \end{abstract}



\author[Zden\v{e}k Mihula]{Zden\v{e}k Mihula}
\address[Zden\v{e}k Mihula]{
Department of Mathematics,
Faculty of Electrical Engineering,
Czech Technical University in Prague,
Technick\'a~2,
166~27 Praha~6,
Czech Republic}
\email{mihulzde@fel.cvut.cz}
\urladdr{0000-0001-6962-7635}

\author[Lubo\v{s} Pick]{Lubo\v{s} Pick}
\address[Lubo\v{s} Pick]{
Department of Mathematical Analysis,
Faculty of Mathematics and Physics,
Charles University,
So\-ko\-lo\-vsk\'a~83,
186~75 Praha~8,
Czech Republic}
\email{pick@karlin.mff.cuni.cz}
\urladdr{0000-0002-3584-1454}

\author[Daniel Spector]{Daniel Spector}
\address[Daniel Spector]{
Department of Mathematics, National Taiwan Normal University, No. 88, Section 4, Tingzhou Road, Wenshan District, Taipei City, Taiwan 116, R.O.C.\\
	\newline
	and
	\newline
	National Center for Theoretical Sciences\\No. 1 Sec. 4 Roosevelt Rd., National Taiwan
	University\\Taipei, 106, Taiwan
	\newline
	and
	\newline
	Department of Mathematics, University of Pittsburgh, Pittsburgh, PA 15261 USA
}

\email{spectda@gapps.ntnu.edu.tw}

\maketitle

\section{Introduction and statement of main results}
Let $\alpha \in (0,n)$ and $1<p<\frac{n}{\alpha}$.  A classical result pioneered by V.~Maz'ya \cites{Mazya1, Mazya2}, extended by D.R.~Adams \cite{Adams:1974}, demonstrated for the full range of parameters by B.J.~Dahlberg \cite{Dahlberg}, whose proof was simplified by K.~Hansson \cite{Hansson}, and has been codified in the literature as \cite[Theorem 7.1.1]{AH} asserts the existence of a constant $C_1=C_1(p,\alpha,n)>0$ such that
\begin{align}\label{capacitary_inequality}
\int_{\mathbb{R}^n} |I_\alpha f|^p \;d{\rm cap}_{\alpha,p} \leq C_1 \|f\|_{L^{p}(\mathbb{R}^n)}^p
\end{align}
for all $f \in L^{p}(\mathbb{R}^n)$.  Here we denote by
\begin{align*}
\text{cap}_{\alpha,p}(E)=\inf\{\|f\|_{L^{p}(\mathbb{R}^{n})}^{p}: f \in \mathcal{S}(\mathbb{R}^n), \quad I_{\alpha}f \geq 1~\text{on}~E\}
\end{align*}
the Riesz capacity, where $I_\alpha f = I_\alpha \ast f$ for 
\begin{align*}
I_{\alpha}(x)=\frac{1}{\gamma(\alpha)} \frac{1}{|x|^{n-\alpha}}, \quad x\in\mathbb{R}^{n},
\end{align*}
 the Riesz kernels, cf.~\cite[p.~117]{Stein}, and the integral on the left-hand-side of \eqref{capacitary_inequality} is intended in the sense of Choquet, i.e.
\begin{align*}
\int_{\mathbb{R}^n} |I_\alpha f|^p \,d{\rm cap}_{\alpha,p} = \int_0^\infty {\rm cap}_{\alpha,p}(\{|I_\alpha f|^p>t\})\dt,
\end{align*}
which can be interpreted as either an improper Riemann integral or Lebesgue integral of the monotone function 
\begin{align*}
t \mapsto {\rm cap}_{\alpha,p}(\{|I_\alpha f|^p>t\}).
\end{align*}

The capacitary or trace inequality \eqref{capacitary_inequality} is a strong form of the Hardy-Littlewood-Sobolev theorem on fractional integration.  For example, when one takes it in conjunction with the isocapacitary inequality (see, e.g. \cite[p.~120]{Stein}),
\begin{align*}
|E|^{1-\alpha p/n} \leq C_2 {\rm cap}_{\alpha,p}(E),
\end{align*}
one deduces the sharp Lorentz version of the Hardy-Littlewood-Sobolev theorem (see also R.~O'Neil \cite[Theorem 2.6 on p.~137]{Oneil}):
\begin{align}\label{lorentz}
\|I_\alpha f\|_{L^{q,p}(\mathbb{R}^n)}^p = q\int_0^\infty |\{|I_\alpha f|>t\}|^{1-\alpha p/n}\;t^{p-1}dt \leq C_3 \|f\|_{L^{p}(\mathbb{R}^n)}^p,
\end{align}
for $q=np/(n-\alpha p)$ and where $L^{q,p}(\mathbb{R}^n)$ is the Lorentz space of functions whose norm given by the left-hand-side of \eqref{lorentz} is finite. Another consequence of \eqref{capacitary_inequality} of equal interest is that it implies, for $f \in L^p(\mathbb{R}^n)$, the existence of Lebesgue points of $I_\alpha f$ not just Lebesgue almost everywhere but up to a set $E$ with ${\rm cap}_{\alpha,p}(E)=0$.  This means that potentials of $L^p(\mathbb{R}^n)$ functions admit Lebesgue points $\mathcal{H}^{n-\alpha p+\epsilon}$ almost everywhere for every $\epsilon \in (0, \alpha p)$.  This last fact follows from a version of \eqref{capacitary_inequality} with the Hardy-Littlewood maximal function on the left-hand-side, the local equivalence of Bessel and Riesz capacities proved in \cite[Proposition 5.1.4 on p.~131]{AH}, and the choice of $h(r)=r^{n-\alpha p +\epsilon}$ in \cite[Theorem 5.1.13 on p.~137]{AH}.  

The inequality \eqref{capacitary_inequality} is slightly weaker than the analogous statement for $p=1$, that one has the bound
\begin{align}\label{adams_inequality}
\int_{\mathbb{R}^n} |I_\alpha f| \; d\mathcal{H}^{n-\alpha}_\infty \leq C_4\|f\|_{\mathcal{H}^1(\mathbb{R}^n)}
\end{align}
for all $f \in \mathcal{H}^1(\mathbb{R}^n)$, the real Hardy space, see e.g. \cite[Proposition 5 on p.~121]{AdamsChoquet}.  Here 
\begin{align*}
\mathcal{H}^{n-\alpha}_\infty(E) = \inf\left \{ \sum_{i=0}^\infty \omega_{n-\alpha} r_i^{n-\alpha} : E \subseteq \bigcup_{i=0}^\infty B(x_i,r_i)\right\}
\end{align*}
denotes the Hausdorff content (see~\cites{COS,CS, Eri:24,Esm:22,Har:23,Har:23b,Her:24,MS, PS1, PS2, PS3, RSS,S, S1} for related results), $\omega_{n-\alpha}= \pi^{(n-\alpha)/2}/\Gamma((n-\alpha)/2+1)$ is a normalization constant, and again the integral is intended in the sense of Choquet.  One says the inequality \eqref{capacitary_inequality} is weaker because while for $1 \leq p < n/\alpha$ the behavior of the capacity on balls and a covering argument easily give
\begin{align}\label{capacity-content}
{\rm cap}_{\alpha,p}(E) \lesssim \mathcal{H}^{n-\alpha p}_\infty(E),
\end{align}
the reverse implication fails unless $p=1$.  Here we use $\text{cap}_{\alpha,1}$ to denote an analogue of the capacity $\text{cap}_{\alpha,p}$ for $p=1$ defined by
\begin{align*}
\text{cap}_{\alpha,1}(E)=\inf\{\|f\|_{\mathcal{H}^1(\rn)}: f \in \mathcal{S}_0(\mathbb{R}^n), \quad I_{\alpha}f \geq 1~\text{on}~E\},
\end{align*}
where $\mathcal{S}_0(\mathbb{R}^n)$ is the subset of $\mathcal{S}(\mathbb{R}^n)$ consisting of Schwartz functions with zero mean value. This failure of the reverse implication can be seen by the Cantor set construction and Theorem 5.3.2 in \cite[p.~142-143]{AH} (which combined with \cite[Proposition 2.3.7]{AH} shows the failure of the analogue of \eqref{adams_inequality} for $p>1$), while the validity at the endpoint is itself a consequence of \eqref{adams_inequality}.  One notes from this strengthening of \eqref{capacitary_inequality} that potentials of functions in the real Hardy space admit Lebesgue points up to a set of $\mathcal{H}^{n-\alpha}$ measure zero, with no loss of $\epsilon>0$.  

The consideration of \eqref{capacitary_inequality} and \eqref{adams_inequality} and their discrepancies might prompt one to wonder whether with stronger assumptions it is possible to prove a complete analogue of the latter in the regime $p>1$\textemdash a trace inequality with respect to the appropriately scaling Hausdorff content.  An answer to this question was given in the remarkable paper of M.~Korobkov and J.~Kristensen \cite{Kristensen-Korobkov}, who proved that for functions in the Lorentz space $L^{p,1}(\mathbb{R}^n)$ one can recover such a trace inequality.  In particular, from their paper on Luzin N- and Morse-Sard properties for a borderline case \cite[Theorem 1.2]{Kristensen-Korobkov} one has
\begin{theorem-A}[Korobkov-Kristensen]
Let $\alpha \in (0,n)$ and $1<p<n/\alpha$.  There exists a constant $C_5=C_5(\alpha,p,n)>0$ such that
\begin{align}\label{kristensen-korobkov}
\int_{\mathbb{R}^n} |I_\alpha f|^p \; d\mathcal{H}^{n-\alpha p}_\infty \leq C_5\|f\|_{L^{p,1}(\mathbb{R}^n)}^p
\end{align}
for all $f \in L^{p,1}(\mathbb{R}^n)$.
\end{theorem-A}
\noindent
Here we recall that the Lorentz space $L^{p,1}(\mathbb{R}^n)$ is the set of measurable $f$ such that the norm
\begin{align*}
\|f\|_{L^{p,1}(\mathbb{R}^n)}= p \int_0^\infty |\{ |f|>t\}|^{1/p}\;dt
\end{align*}
is finite. Note that a different definition of Lorentz (quasi)norms is used in the rest of the paper (see Section~\ref{sec:prel}), but it coincides with the one used here (see~\cite[Proposition 1.4.9 on p.~53]{Grafakos}, for example).

\begin{remark}\label{rem:KK-duality-explained}
While \cite[Theorem 1.2]{Kristensen-Korobkov} asserts an estimate with respect to measures in a Morrey space, a duality argument shows that their formulation is equivalent to that in Theorem~A. For the reader's convenience, we briefly explain the duality argument here. The density of $C_c(\mathbb{R}^n)$ in $L^{p,1}(\mathbb{R}^n)$ yields $\mathcal{H}^{n-\alpha p}_\infty$-quasicontinuity of the potential $I_\alpha f$ and so also of $|I_\alpha f|^p$. By the duality \cite[Proposition~1 on p.~118]{AdamsChoquet} (see also \cite{ST:22}) between $L^1(\mathcal{H}^{n-\alpha p}_\infty)$ and the Morrey space $L^{1, n-\alpha p}$, consisting of those (signed) Radon measures $\mu$ on $\rn$ for which
\begin{equation*}
\|\hspace{-0.09em}| \mu \|\hspace{-0.09em}| = \sup_{Q} \frac{\mu(Q)}{l(Q)^{n-\alpha p}}  < \infty,
\end{equation*}
using the Hahn-Banach theorem one obtains that
\begin{equation}\label{rem:KK-duality-explained:eq}
\int_{\mathbb{R}^n} |I_\alpha f|^p \; d\mathcal{H}^{n-\alpha p}_\infty \approx \sup_{\|\hspace{-0.09em}| \mu \|\hspace{-0.09em}|\leq1}\int_{\rn} |I_\alpha f|^p \intdif{|\mu|}
\end{equation}
for every $f\in L^{p,1}(\mathbb{R}^n)$. Here $Q\subseteq \rn$ is a cube and $l(Q)$ denotes the length of $Q$. One can replace the $\approx$ with equality in this statement about Banach space duality by utilizing a norm on $L^1(\mathcal{H}^{n-\alpha p}_\infty)$ in place of the above quasi-norm, for example with a Choquet integral involving the corresponding dyadic Hausdorff content. Finally, while Theorem~A asserts that the left-hand side of~\eqref{rem:KK-duality-explained:eq} is bounded from above by a constant multiple of $\|f\|_{L^{p,1}(\mathbb{R}^n)}^p$, \cite[Theorem 1.2]{Kristensen-Korobkov} asserts the same for the right-hand side of \eqref{rem:KK-duality-explained:eq}. Therefore, the formulations are indeed equivalent.
\end{remark}

The proof of Korobkov and Kristensen is in two steps. First, they utilize a fundamental property of the space $L^{p,1}(\mathbb{R}^n)$, that to establish \eqref{kristensen-korobkov} it suffices to demonstrate the inequality for characteristic functions of sets of finite measure (see, e.g. \cite[Theorem 3.13 on p.~195]{SteinWeiss}). Second, they prove a series of  lemmas \cite[Lemmas 3.1-3.7]{Kristensen-Korobkov} which establishes such an inequality by elementary arguments, namely the following theorem.
\begin{theorem-B}[Korobkov-Kristensen]
Let $\alpha \in (0,n)$ and $1<p<n/\alpha$.  There exists a constant $C_6=C_6(\alpha,p,n)>0$ such that
\begin{align}\label{kristensen-korobkov-prime}
\int_{\mathbb{R}^n} |I_\alpha \chi_E|^p \; d\mathcal{H}^{n-\alpha p}_\infty \leq C_6|E|
\end{align}
for all measurable $E \subseteq \mathbb{R}^n$ such that $|E|<+\infty$.
\end{theorem-B}

The starting point of this paper is an observation concerning the connection of the result of Korobkov and Kristensen and a classical result of E. Sawyer \cites{Sawyer1,Sawyer2}, combined with the same duality principle as in Remark~\ref{rem:KK-duality-explained}. To this end, we recall that in his papers on one and two weight estimates, Sawyer proved (see also~\cite[pp.~28--29]{Ada:98}) the following:  For $\beta \in (0,n)$ and $1<q<\frac{n}{\beta}$ there exists a constant $C_7=C_7(\beta,q,n)>0$ such that
\begin{align}\label{adams_fractional}
\int_{\mathbb{R}^n} \left(\mathcal{M}_\beta f\right)^q\;d\mathcal{H}^{n-\beta q}_\infty \leq C_7\|f\|_{L^q(\mathbb{R}^n)}^q
\end{align}
for all $f \in L^q(\mathbb{R}^n)$, where 
\begin{align*}
\mathcal{M}_\beta(f)(x) = \sup_{r>0} \frac{1}{\omega_{n-\beta} r^{n-\beta}} \int_{B(x,r)} |f(y)|\dy.
\end{align*}
For $\beta=\alpha$ and $q=p$, the inequality \eqref{capacity-content} shows that the estimate for the fractional maximal function in \eqref{adams_fractional} is better than that for the Riesz potential of the same order in \eqref{capacitary_inequality}, while, as discussed in the preceding, there is no hope to control the integral of the Riesz potential with respect to the Hausdorff content of this order.  

Yet we observe that there is a classical inequality for Riesz potentials which allows one to find room in the inequality because of the stronger hypothesis in Theorem~A.  In particular, \cite[Proposition 3.1.2(c) on p.~54]{AH} asserts that for $\beta \in (0,n)$ and $\theta \in (0,1)$, there exists a constant $C_8=C_8(\beta,\theta,n)>0$ such that 
\begin{align}\label{interpolation-Riesz}
|I_{\theta \beta} f(x) | \leq C_8 \mathcal{M}_\beta f(x)^\theta \mathcal{M} f(x)^{1-\theta}.
\end{align}
Thus, for any $\alpha \in (0,n)$ and $1<p<\frac{n}{\alpha}$, one may choose $\beta\in (\alpha,n)$ such that $q=p \frac{\alpha}{\beta}>1$.  Then, for $f=\chi_E$ and $\theta=\alpha/\beta$, \eqref{interpolation-Riesz} yields
\begin{align*}
|I_\alpha \chi_E(x)| \leq  C_8 \mathcal{M}_\beta \chi_E(x)^{\alpha/\beta}.
\end{align*}
As a consequence of this inequality and the relation $\alpha p = \beta q$, one deduces
\begin{align*}
\int_{\mathbb{R}^n} |I_\alpha \chi_E|^p \; d\mathcal{H}^{n-\alpha p}_\infty &\leq C_8^p \int_{\mathbb{R}^n} |\mathcal{M}_\beta \chi_E|^q \; d\mathcal{H}^{n-\beta q}_\infty \\
&\leq C_7 C_8^p |E|,
\end{align*}
which is the inequality \eqref{kristensen-korobkov-prime}.  Thus one finds a short proof of Theorem B, and therefore Theorem A, on the basis of the powerful inequality \eqref{adams_fractional}.  

This approach seems to have gone unnoticed until now, despite the great interest in trace inequalities, and is a key idea for the new results we establish in this paper.  In particular, in this work we develop a framework for this principle, when a bound for a good operator like the fractional maximal operator can easily be translated into a bound for a (comparatively) bad operator like the Riesz potential.  To demonstrate the new theory we develop with a concrete example, we record here the following generalization of Korobkov and Kristensen's result.

\begin{theorem}\label{mainresult-prime}
Let $0<d \leq n$, $\alpha \in (0,d)$, and $1<p<\frac{d}{\alpha}$.  There exists a constant $C_9=C_9(\alpha,p,d,n)>0$ such that
\begin{align}\label{kristensen-korobkov-two-measures}
\int_{\mathbb{R}^n} |I_\alpha^\mu f|^p \intdif{\nu} \leq C_9\|f\|_{L^{p,1}(\mathbb{R}^n,\mu)}^p
\end{align}
for all $f \in L^{p,1}(\mathbb{R}^n, \mu)$ and for all Radon measures $\mu, \nu$ which satisfy
\begin{equation}\label{E:mainresult-prime:d_upper_ahlfors}
\sup_{Q} \frac{\mu(Q)}{l(Q)^{d}} < \infty
\end{equation}
and
   \begin{equation}\label{nu_bound}
       \sup_{\mu(Q)>0} \frac{\nu(Q)}{\mu(Q)^{1-\frac{\alpha p}{d}}} < \infty.
    \end{equation}
Here the Riesz potential $I_\alpha^\mu$ of order $\alpha \in (0,d)$ with respect to $\mu$ is defined by the formula
\begin{align}\label{generalized_potential_def}
I_\alpha^\mu f(x) = \int_{\mathbb{R}^n} \frac{f(y)}{|x-y|^{d-\alpha}}\intdif{\mu(y)},\ x\in\rn.
\end{align}
\end{theorem}
\noindent
When $\mu$ is the Lebesgue measure on $\rn$ (and so $d = n$), this recovers the Korobkov--Kristensen trace inequality~\eqref{kristensen-korobkov} by a different argument. In particular, our Theorem \ref{mainresult-prime} and the argument presented in Remark  \ref{rem:KK-duality-explained} yield the assertions of Theorems~A/B as special cases of our results.  In this case, and in general with \eqref{E:mainresult-prime:d_upper_ahlfors} as given, the assumption \eqref{nu_bound} is necessary, which can be seen by taking $f=\chi_Q$.  However, the combination of \eqref{E:mainresult-prime:d_upper_ahlfors} and \eqref{nu_bound} need not be necessary and may be relaxed to the imposition of the testing condition of E. Sawyer in~\cite{Sawyer1} which our work builds upon.

Theorem \ref{mainresult-prime} provides a specific example of the new results which follow from our work, though our results hold in the broader context of rearrangement-invariant function spaces (see Theorem \ref{thm:most_general_KK}).  This preempts the question of bounds for bad operators with that of necessary and sufficient conditions for bounds for good operators on these spaces. This general context of rearrangement-invariant function spaces provides a unifying theory for function spaces such as Lebesgue spaces, Lorentz spaces, or Orlicz spaces, to name a few. 

The classical Calder\'on-type theorem (see \cite{BR}, \cite[Theorem 5.7 on p.~144]{BS}) asserts that restricted weak-type boundedness of linear operators is equivalent to the same boundedness for quasi-linear operators and they are equivalent to the boundedness of the Calder\'on operator. However, for the example of the fractional maximal function, the bounds are not of restricted weak type. Instead, one has the pair of estimates
\begin{align}
    \mathcal{M}_{\alpha}&\colon L^{p}(\rn, dx) \to L^{p}(\rn, \nu) \label{intro:frac_max_p_endpoint},
            \\
     \mathcal{M}_{\alpha}&\colon L^{\frac{n}{\alpha},\infty}(\rn, dx) \to L^{\infty}(\rn, \nu), \label{intro:frac_max_to_linfty_endpoint}
  \end{align}
  where $\nu$ is a Radon measure on $\rn$ that satisfies
  \begin{equation*}
     \sup_{Q} \frac{\nu(Q)}{l(Q)^{n - \alpha p}} < \infty.
  \end{equation*}
Recall that the corresponding restricted weak-type estimates for a quasi-linear operator $T$ read as:
\begin{align*}
&T\colon L^{p,1}(\rn, dx) \to L^{p,\infty}(\rn, \nu), \\
&T\colon L^{\frac{n}{\alpha}, 1}(\rn, dx) \to L^\infty(\rn, \nu).
\end{align*}
When $T=\mathcal{M}_\alpha$, the conclusion of the classical Calder\'{o}n theorem is not optimal, because it does not fully exploit the endpoint boundedness properties \eqref{intro:frac_max_p_endpoint} and \eqref{intro:frac_max_to_linfty_endpoint}, which are better than of restricted weak-type.
  As we will see, this results in significant differences in the theory we develop from the classical one. Several other types of various nonstandard versions of Calder\'on's theorem can be found in literature, see e.g.~\cites{Bae:22,GP-Indiana:09,Mal:12}, but none of the known ones can be used for our purposes.

We therefore next introduce a class of operators inspired by the results on fractional maximal operators mentioned above for which we will establish such a Calder\'on-type theorem. The symbol $\M$ denotes the class of measurable functions on a given measure space, and $\M_0$ denotes those that are finite almost everywhere.

\begin{definition}
Let $(\RR,\mu)$ and $(\cS, \nu)$ be nonatomic $\sigma$-finite measure spaces. 
Let $p,\pone$ be such that 
\begin{equation}\label{E:pq}
    1<p<\pone.
\end{equation}
We say that an operator $T$ defined on $(L^{p} + L^{\pone , \infty})(\RR, \mu)$ and taking values in $\M_0(\cS, \nu)$ is \emph{$(p,\pone)$-sawyerable if $T$  is a quasi-linear and}  satisfies:
\begin{align}
    T&\colon L^{p}(\RR, \mu) \to L^{p}(\cS, \nu) \label{saw_op_p_p_endpoint}\\
    \intertext{and}
     T&\colon L^{\pone ,\infty}(\RR, \mu) \to L^{\infty}(\cS, \nu). \label{saw_op_p1_infty_endpoint}
\end{align}
\end{definition}

Recall that an operator $T$ defined on a linear space $X\subseteq \M_0(\RR,\mu)$ and taking values in $\M_0(\cS,\nu)$ is \emph{quasi-linear} if there is a constant $k\geq1$ such that
\begin{equation*}
    |T(f + g)|\leq k\big( |Tf| + |Tg|  \big) \quad \text{and} \quad |T(\alpha f)| = |\alpha| |Tf| \quad \text{$\nu$-a.e.~in $\cS$}
\end{equation*}
for every $f,g\in X$ and every scalar $\alpha$. 
We say that $T$ is \emph{sublinear} if it is quasi-linear with $k = 1$.

We shall now point out that sawyerable operators can be effectively characterized by a special governing operator acting on (nonincreasing) functions of a single variable.

Let $p,\pone$ satisfy~\eqref{E:pq}, and let $\ptwo$ be defined by
\begin{equation}
\label{def:sawyerable_op:def_of_p}
    \ptwo = \frac{\pone }{\pone-p}.
\end{equation}
We then define the operator~$R_{p,\pone}$ by
\begin{equation}
\label{def:sawyerable_op:def_of_R}
     R_{p,\pone}g(t) = \Big( \frac1{t} \int_0^{t^{\ptwo }} g^*(s)^p \ds \Big)^\frac1{p},\ t\in (0, \infty),\ g\in\M(0, \infty),
\end{equation}
where $g^*$ is the nonincreasing rearrangement of $g$. The operator $R_{p,\pone}$ is intimately connected with the $K$-inequality suitable for the pair of estimates~\eqref{intro:frac_max_p_endpoint}--\eqref{intro:frac_max_to_linfty_endpoint}, and its origin will be apparent from Proposition~\ref{prop:joint_sawyerable_vs_sep_sawyerable} below. Its importance stems from the following two theorems, the first one being in the spirit of classical theorems of Calder\'on. An additional principal novelty is the appearance of the space $Y^{\langle p \rangle}(\cS,\nu)$, governed by the functional $\|g\|_{Y^{\langle p \rangle}}=\|(|g|^p)^{**}(t)^{\frac1p}\|_Y$, where $Y(\cS,\nu)$ is a rearrangement-invariant space. While the precise definitions are postponed to Section~\ref{sec:prel}, the two abstract theorems are followed by illustrating examples.

\begin{theorem}\label{thm:calderon_thm_for_sawyerable_op}
Let $(\RR,\mu)$ and $(\cS, \nu)$ be nonatomic $\sigma$-finite measure spaces. Let $p,\pone$ satisfy~\eqref{E:pq}.  Then, for every couple $X(\RR,\mu)$ and $Y(\cS, \nu)$ of rearrangement-invariant function spaces, where $X(\RR, \mu)\subseteq (L^{p} + L^{\pone , \infty})(\RR, \mu)$, the following three statements are equivalent:
\begin{enumerate}[(i)]
    \item Every linear $(p,\pone)$-sawyerable operator $T$ is bounded from $X(\RR,\mu)$ to $Y^{\langle p \rangle}(\cS,\nu)$.
    \item Every quasi-linear $(p,\pone)$-sawyerable operator $T$ is bounded from $X(\RR,\mu)$ to $Y^{\langle p \rangle}(\cS,\nu)$.
    \item The operator $R_{p,\pone}$, defined by \eqref{def:sawyerable_op:def_of_R} with $\ptwo$ from~\eqref{def:sawyerable_op:def_of_p}, is bounded from $\bar X(0, \mu(\RR))$ to $\bar Y(0, \nu(\cS))$,
\end{enumerate}
in which $\bar X(0, \mu(\RR))$ and $\bar Y(0, \nu(\cS))$ are representation spaces of $X(\RR,\mu)$ and $Y(\cS, \nu)$, respectively.
\end{theorem}

The characterization of boundedness of sawyerable operators enables us to obtain a variety of boundedness results of a given `bad' operator, whose `good' friend is sawyerable.

\begin{theorem}\label{T:general-principle}
Let $(\RR, \mu)$ and $(\cS, \nu)$ be nonatomic $\sigma$-finite measure spaces. Let $X(\RR,\mu)$ and $Y(\mathcal S,\nu)$ be rearrangement-invariant function spaces. Assume that
    \begin{equation}\label{T:general-principle:fundX_vanishes}
        \lim_{t\to 0^+} \|\chi_{(0,t)}\|_{\bar X(0, \mu(\RR))} = 0.
    \end{equation}
    Let $p,\pone$ satisfy~\eqref{E:pq}, and let $\lambda\in(0,1)$ be such that
    \begin{equation}\label{E:parameters-large-enough}
        \lambda p>1.
    \end{equation}
    Let $B$ be a linear operator defined on characteristic functions of $\mu$-measurable subsets of $\RR$ of finite measure and taking values in $\M_0(\cS, \nu)$. Let $G$ be an order preserving $(\lambda p,\lambda \pone)$-sawyerable operator. Assume that $R_{p,\pone}\colon \bar X(0, \mu(\RR)) \to \bar Y(0,\nu(\cS))$, and that there is a constant $C>0$ such that for every $\mu$-measurable set $E\subseteq\RR$ of finite measure: 
    \begin{equation}\label{E:pointwise-general}
        |B\chi_E(y)|
        \leq C |G\chi_E(y)|^{\lambda}
        \quad\text{for $\nu$-a.e.~$y\in\cS$}.
    \end{equation}
    Then
    \begin{equation*}
        B\colon \Lambda_X(\RR,\mu)\to Y^{\langle p\rangle}(\cS,\nu).
    \end{equation*}
\end{theorem}
\noindent
We recall that an operator $T$ is said to be order preserving if $0 \le f \le g$ $\mu$-a.e. implies $0 \leq T(f)\leq T(g)$ $\nu$-a.e.

Loosely speaking, the technical assumption \eqref{T:general-principle:fundX_vanishes} ensures that the space $X$ does not have an ``$L^\infty$ part''. The space $\Lambda_X(\RR,\mu)$ is the classical Lorentz endpoint space corresponding to $X$, for a detailed definition see Section~\ref{sec:prel}. Let us still recall that when $X$ is $L^p$ (or more generally a Lorentz space $L^{p,q}$) with $p\in(1, \infty)$ (and $q\in[1,\infty]$), then $\Lambda_X$ is the Lorentz space $L^{p,1}$, and \eqref{T:general-principle:fundX_vanishes} is satisfied. It is worth noticing that $L^{p,1}$ is precisely the function space appearing on the right-hand side of \eqref{kristensen-korobkov}.

Having stated two abstract theorems, it is in order to illustrate their usage on some practical examples. We first address the question of boundedness of the operator $R_{p,\pone}$ on appropriate function spaces, as one of the key ingredients of the theory. The following theorem characterizes when $R_{p,\pone}$ is bounded between two Lorentz spaces (and so also between two Lebesgue spaces).

\begin{theorem}
\label{thm:R_boundedness_lorentz_infty}
    Let $r_1,r_2,s_1,s_2\in[1,\infty]$ satisfy \eqref{prel:lorentz_ri_cond} with $p=r_j$ and $q=s_j$, $j=1,2$. Let $p,\pone$ satisfy~\eqref{E:pq} and let $r$ be defined by~\eqref{def:sawyerable_op:def_of_p}. Then
    \begin{equation}\label{E:R-on-lorentz}
        R_{p,\pone}\colon 
            L^{r_1,s_1}(0,\infty)\to L^{r_2,s_2}(0, \infty)
    \end{equation}
		if and only if
    \begin{align}
        &r_1\in(p,\pone ),\quad s_1\le s_2,\quad\text{and}\quad\frac{1}{\pone }+\frac{1}{\ptwo r_2}=\frac{1}{r_1},
        \label{E:R-lorentz-parameters-i}
    \end{align}  
    or
    \begin{align}
        &r_1=p,\quad r_2=p,\quad s_1\le p,\quad\text{and}\quad s_2=\infty,
        \label{E:R-lorentz-parameters-ii}
    \end{align}  
    or
    \begin{align}
        &r_1=\pone ,\quad r_2=\infty\quad\text{and}\quad s_2=\infty.
        \label{E:R-lorentz-parameters-iii}
    \end{align}
\end{theorem}

Now when the boundedness properties of $R_{p,\pone}$ between Lorentz spaces is at our disposal, the next step is to investigate what the operation $(\cdot)^{\langle p\rangle}$ does on them. By~\cite[Examples~4.7 and 4.9]{Tur:23}, we have
\begin{equation}\label{ex:sawyerable_T_boundedness_lorentz_infty:power}
    (L^{r,s})^{\langle p\rangle}(\cS, \nu) = \begin{cases}
        L^{r,s}(\cS, \nu)\quad&\text{when $p<r<\infty$ and $s\in[1,\infty]$},\\
        L^p(\cS,\nu) \quad&\text{when $r=p$ and $s=\infty$},\\
				L^\infty(\cS,\nu) \quad&\text{when $r=s=\infty$}.
    \end{cases}
\end{equation}
What is particularly important here is that the operation enhances $L^{p,\infty}$ to $L^p$, which in turn leads to stronger estimates for sawyerable operators.

The strength and generality of Theorems \ref{thm:calderon_thm_for_sawyerable_op} and \ref{T:general-principle} can be fully understood in detail by the applications discussed at the end of Section~\ref{sec:calderon_general_theory}. Their important application involving Lorentz spaces, which generalizes Theorem~A, is our Theorem \ref{mainresult-prime} stated above. In fact, it is a corollary of a more general theorem, namely Theorem~\ref{thm:most_general_KK}, which establishes a general boundedness result for the Riesz potential $I_\alpha^\mu$, defined by \eqref{generalized_potential_def}, under suitable assumptions on a pair of Radon measures $\mu, \nu$. It relies on our principle to connect a bad operator, $I_\alpha^\mu$, to a good one, for which bounds are known. Here the good operator is a suitable dyadic maximal function associated to the measure $\mu$ for which bounds were established by E. Sawyer in~\cite{Sawyer1} (see the discussion before Lemma~\ref{bigger_lemma} for more details).

\section{Preliminaries}\label{sec:prel}
In the entire paper, we use the convention that $0\cdot \infty = 0$. We use the symbol $\lesssim$ in inequalities to mean that the left-hand side is less than or equal to a constant multiple of the right-hand side, with the multiplicative constant independent of all important quantities. When it is not obvious from the context what the important quantities are, we explicitly state it. Loosely speaking, the multiplicative constant may depend on parameters of function spaces (such as $p$ in the case of $L^p$ spaces) and on the measure of the underlying measure space in the case of finite measure spaces. We also use the symbol $\approx$ when $\lesssim$ and $\gtrsim$ hold simultaneously, where $\gtrsim$ substitutes for $\lesssim$ with switched sides.

Let $(\RR,\mu)$ be a nonatomic measure space. The set of all $\mu$-measurable functions on $\RR$ is denoted by $\M(\RR, \mu)$. We denote by $\M^+(\RR, \mu)$ and $\M_0(\RR, \mu)$ its subset consisting of those functions that are nonnegative and finite $\mu$-a.e., respectively. We say that functions $f\in\M(\RR, \mu)$ and $g\in\M(\cS, \nu)$, where $(\cS, \nu)$ is another (possibly different) measure space, are \emph{equimeasurable} if their distributional functions coincide, that is,
\begin{equation*}
    \mu(\{x\in\RR: |f(x)| > \lambda \}) = \nu(\{y\in\cS: |g(y)| > \lambda\}) \quad \text{for every $\lambda > 0$}.
\end{equation*}

The \emph{nonincreasing rearrangement} of a function $f\in\M(\RR, \mu)$ is the function $f^*_\mu\colon (0, \infty) \to [0, \infty]$ defined as
\begin{equation*}
    f^*_\mu(t) = \inf\{\lambda>0: \mu(\{x\in\RR: |f(x)| > \lambda\})\leq t\},\ t\in(0, \infty).
\end{equation*}
It clearly follows from the definition that 
\begin{equation}\label{E:lattice}
    \text{if $|f|\le|g|$ $\mu$-a.e., then  $f^*_\mu\le g^*_\mu$}.
\end{equation}
The nonincreasing rearrangement is nonincreasing and right-continuous. Moreover, $f$ and $f^*_\mu$ are equimeasurable, and $f^*_\mu$ vanishes in the interval $[\mu(\RR), \infty)$. The \emph{maximal nonincreasing rearrangement} of a function $f\in\M(\RR, \mu)$ is the function $f^{**}_\mu\colon (0, \infty) \to [0, \infty]$ defined as
\begin{equation*}
    f^{**}_\mu(t) = \frac1{t}\int_0^t f^*_\mu(s) \ds,\ t\in(0, \infty).
\end{equation*}
The maximal nonincreasing rearrangement is nonincreasing and continuous. Moreover, it dominates the nonincreasing rearrangement, i.e., $f^*_\mu\leq f^{**}_\mu$. The maximal nonincreasing rearrangement satisfies (see~\cite[Chapter 2, Proposition~3.3]{BS})
\begin{equation}\label{prel:double_star_alt_expr}
    f^{**}_\mu(t) = \frac1{t}\sup_{E\subseteq \RR, \mu(E) = t} \int_E |f(x)| \intdif{\mu(x)} \quad \text{for every $t\in(0, \mu(\RR))$}.
\end{equation}
A special case of the \emph{Hardy--Littlewood inequality} tells us that
\begin{equation}\label{prel:HL_ineq}
    \int_E |f(x)| \intdif{\mu(x)} \leq \int_0^{\mu(E)} f^*_\mu(t) \dt \quad \text{for every $f\in\M(\RR, \mu)$}
\end{equation}
and every $\mu$-measurable $E\subseteq\RR$.

A functional $\|\cdot\|_{X(\RR, \mu)}\colon  \M^+ (\RR, \mu) \to [0,\infty]$ is called a \emph{rearrangement-invariant Banach function norm} if, for all $f$, $g$ and $\{f_j\}_{j\in\N}$ in $\M^+(\RR, \mu)$, and every $\lambda \geq0$, the following properties hold:
\begin{enumerate}[(P1)]
\item $\|f\|_{X(\RR, \mu)} = 0$ if and only if $f=0$ $\mu$-a.e.;
$\|\lambda f\|_{X(\RR, \mu)}= \lambda \|f\|_{X(\RR, \mu)}$; $\|f+g\|_{X(\RR, \mu)}\leq \|f\|_{X(\RR, \mu)} + \|g\|_{X(\RR, \mu)}$;
\item $  f \le g$ $\mu$-a.e.\ implies $\|f\|_{X(\RR, \mu)}\leq\|g\|_{X(\RR, \mu)}$;
\item $  f_j \nearrow f$ $\mu$-a.e.\ implies
$\|f_j\|_{X(\RR, \mu)} \nearrow \|f\|_{X(\RR, \mu)}$;
\item $\|\chi_E\|_{X(\RR, \mu)} < \infty$ for every $E\subseteq \RR$ of finite measure;
\item  if $E\subseteq \RR$ is of finite measure, then $\int_{E} f\intdif{\mu(x)} \le C_E
\|f\|_{X(\RR, \mu)}$, where $C_E$ is a positive constant possibly depending on $E$ and $\|\cdot\|_{X(\RR, \mu)}$ but not on $f$;
\item $\|f\|_{X(\RR, \mu)} = \|g\|_{X(\RR, \mu)}$ whenever $f$ and $g$ are equimeasurable.
\end{enumerate}
We extend $\|\cdot\|_{X(\RR, \mu)}$ to all functions $f\in\M(\RR, \mu)$ by defining
\begin{equation*}
    \|f\|_{X(\RR, \mu)} = \|\,|f|\,\|_{X(\RR, \mu)},\ f\in\M(\RR, \mu).
\end{equation*}
The functional $\|\cdot\|_{X(\RR, \mu)}$ is a norm on the set
\begin{equation*}
X(\RR, \mu) = \{f\in\M(\RR, \mu)\colon \| f \|_{X(\RR, \mu)}<\infty\},
\end{equation*}
with the convention that we identify functions which agree $\mu$ almost everywhere.
In fact, $X(\RR, \mu)$ endowed with $\|\cdot\|_{X(\RR, \mu)}$ is a Banach space, which is contained in $\M_0(\RR, \mu)$. We will call $X(\RR, \mu)$ a \emph{rearrangement-invariant function space}. When $(\RR, \mu)$ is an interval $(0, a)$ endowed with the Lebesgue measure, where $a\in(0, \infty]$, we write $X(0, a)$ for the sake of simplicity, and we will also omit the subscript in the notation of rearrangements.

When $X(\RR, \mu)$ and $Y(\RR, \mu)$ are two rearrangement-invariant function spaces, $X(\RR, \mu) \subseteq Y(\RR, \mu)$ means that there is a constant $C>0$ such that
\begin{equation*}
    \|f\|_{Y(\RR,\mu)} \leq C \|f\|_{X(\RR, \mu)} \quad \text{for every $f\in\M(\RR, \mu)$}.
\end{equation*}
By $X(\RR, \mu) = Y(\RR, \mu)$, we mean that $X(\RR, \mu) \subseteq Y(\RR, \mu)$ and $Y(\RR, \mu) \subseteq X(\RR, \mu)$ simultaneously. In other words, the rearrangement-invariant function spaces coincide up to equivalent norms.

Given a rearrangement-invariant function space $X(\RR, \mu)$, its \emph{representation space} is the unique rearrangement-invariant function space $\bar X(0, \mu(\RR))$
representing $X(\RR, \mu)$ in the sense that (see~\cite[Chapter~2, Theorem~4.10]{BS})
\begin{equation*}
	\|f\|_{X(\RR, \mu)} = \|f^*_\mu\|_{\bar X(0, \mu(\RR))} \quad \text{for every $f\in\M(\RR, \mu)$}.
\end{equation*}
Note that $X(0, a) = \bar X(0, a)$ for every $a\in(0, \infty]$, and $X(\RR, \mu) \subseteq Y(\RR, \mu)$ if and only if $\bar X(0, \mu(\RR)) \subseteq \bar Y(0, \mu(\RR))$.

Textbook examples of rearrangement-invariant function spaces are the Lebesgue spaces $L^p(\RR, \mu)$, $p\in[1,\infty]$. Their rearrangement invariance follows from the layer cake representation formula (e.g., see~\cite[Theorem~1.13]{LL:01}). More precisely, we have
\begin{equation*}
    \|f\|_{L^p(\RR, \mu)} = \|f^*_\mu\|_{L^p(0, \mu(\RR))} \quad \text{for every $f\in\M(\RR, \mu)$}.
\end{equation*}
Lorentz spaces and Orlicz spaces are other important and well-known examples of rearrangement-invariant function spaces. In this paper, apart from Lebesgue spaces, we also work with Lorentz spaces $L^{p,q}(\RR,\mu)$, and so we briefly introduce them here. The functional $\|\cdot\|_{L^{p,q}(\RR, \mu)}$ defined as
\begin{equation*}
    \|f\|_{L^{p,q}(\RR, \mu)} = \|t^{\frac1{p} - \frac1{q}} f^*_\mu(t)\|_{L^q(0, \mu(\RR))},\ f\in\M(\RR, \mu),
\end{equation*}
is a rearrangement-invariant Banach function norm if and only if $1\leq q \leq p < \infty$ or $p=q=\infty$. When $1<p<q\leq \infty$, it satisfies all the properties of a rearrangement-invariant Banach function norm except~(P1) (more precisely, the functional is not subadditive). However, it is still at least equivalent to a rearrangement-invariant Banach function norm even when $1<p<q\leq \infty$\textemdash the norm is defined in the same way but with $f^*_\mu$ replaced by $f^{**}_{\mu}$. As we will not be interested in precise values of constants, we will consider $L^{p,q}(\RR,\mu)$ a rearrangement-invariant function space whenever
\begin{equation}\label{prel:lorentz_ri_cond}
    p=q=1\quad\text{or}\quad p\in(1, \infty)\ \text{and}\ q\in[1,\infty],\quad\text{or}\quad p=q=\infty.
\end{equation}
In the remaining part of the paper, it will be implicitly assumed that the parameters $p,q$ satisfy \eqref{prel:lorentz_ri_cond}. Moreover, we have
\begin{equation}\label{prel:two_stars_on_Lorentz_spaces}
    \|f^{**}\|_{L^{p,q}(0,\infty)} \leq p' \|f\|_{L^{p,q}(0, \infty)} \quad \text{for every $f\in\M(0, \infty)$}
\end{equation}
provided that $p>1$ (e.g.,~see~\cite[Chapter~4, Lemma~4.5]{BS}).
Note that $L^{p,p}(\RR, \mu) = L^p(\RR, \mu)$ (in fact, they have the same norms). The Lorentz spaces $L^{p,\infty}(\RR, \mu)$ are often called weak Lebesgue spaces. Lorentz spaces are increasing with respect to the second parameter, i.e.,
\begin{equation}\label{prel:lorentz_second_index_increasing}
    L^{p,q_1}(\RR, \mu)\subseteq L^{p,q_2}(\RR, \mu) \quad \text{when $q_1\leq q_2$}.
\end{equation}
Furthermore, note that $\bar X(0, \mu(\RR)) = L^{p,q}(0, \mu(\RR))$ when $X(\RR, \mu) =  L^{p,q}(\RR, \mu)$ (possibly up to equivalent norms).

The \emph{fundamental function} of a rearrangement-invariant function space $X(\RR, \mu)$ is the function $\varphi_X\colon(0, \mu(\RR)) \to (0, \infty)$ defined as
\begin{equation*}
    \varphi_X(t) = \|\chi_{(0,t)}\|_{\bar X(0, \mu(\RR))},\ t\in(0, \mu(\RR)).
\end{equation*}
Notice that $\varphi_X(t) = \|\chi_E\|_{X(\RR, \mu)}$, where $E\subseteq\RR$ is any subset of $\RR$ satisfying $\mu(E) = t$. For example, $\varphi_{L^p}(t) = t^\frac1{p}$. More generally, $\varphi_{L^{p,q}}(t) \approx t^\frac1{p}$.

Given a rearrangement-invariant function space $X(\RR, \mu)$, we define the functional $\|\cdot\|_{\Lambda_X(\RR, \mu)}$ as
\begin{equation*}
    \|f\|_{\Lambda_X(\RR, \mu)} = \|f\|_{L^\infty(\RR, \mu)} \varphi_X(0^+) + \int_0^\infty f^*_\mu(s) \varphi_X'(s) \ds,\ f\in\M(\RR, \mu).
\end{equation*}
The functional $\|\cdot\|_{\Lambda_X(\RR, \mu)}$ is a rearrangement-invariant function norm provided that $\varphi_X$ is concave. The fundamental function of a rearrangement-invariant function space is quasiconcave but it need not be concave in general. If $\varphi_X$ is only quasiconcave, then the functional $\|\cdot\|_{\Lambda_X(\RR, \mu)}$ is not necessarily subadditive (cf.~\cite{L:51}). However, there always is an equivalent rearrangement-invariant function norm on $X(\RR, \mu)$ with respect to which the fundamental function is concave. The space $\Lambda_X(\RR, \mu)$ is contained in $X(\RR, \mu)$, and their fundamental functions  coincide (possibly up to multiplicative constants). For example,
\begin{equation}\label{prel:Lorentz_spaces_Lambda_endpoint}
    \Lambda_{L^{p,q}}(\RR, \mu) = L^{p,1}(\RR, \mu)
\end{equation}
provided that $p<\infty$, and $\Lambda_{L^\infty}(\RR, \mu) = L^\infty(\RR, \mu)$. Furthermore, there is also a largest rearrangement-invariant function space with the same fundamental function as $X(\RR, \mu)$, which is equivalent to $L^{p,\infty}(\RR, \mu)$ for $X(\RR, \mu)=L^{p,q}(\RR, \mu)$ with $p>1$. The interested reader can find more information in \cite[Chapter~2, Section~5]{BS} (for spaces endowed with norms) and also in~\cite{Nek:24}  (for spaces endowed with quasinorms).

Given $\alpha>0$ and a rearrangement-invariant function space $X(\RR,\mu)$, the function space $X^{\langle \alpha \rangle}(\RR,\mu)$ is defined as the collection of all $f\in\M(\RR, \mu)$ such that
$\|f\|_{X^{\langle \alpha \rangle}(\RR,\mu)}<\infty$, where 
\begin{equation*}
    \|f\|_{X^{\langle \alpha \rangle}(\RR,\mu)}
    =    \big\|\big((|f|^{\alpha})_\mu^{**}\big)^{\frac{1}{\alpha}}\big\|_{\bar X(0, \mu(\RR))}.
\end{equation*}
When $\mu(\RR) < \infty$, $X^{\langle \alpha \rangle}(\RR,\mu)$ is always a rearrangement-invariant function space. When $\mu(\RR) = \infty$, $X^{\langle \alpha \rangle}(\RR,\mu)$ is a rearrangement-invariant function space unless it is trivial (i.e., it contains only the zero function, and so it does not satisfy (P4)), which may happen. For a detailed study of the spaces $X^{\langle \alpha \rangle}(\RR,\mu)$, see~\cite{Tur:23} (recall also~\eqref{ex:sawyerable_T_boundedness_lorentz_infty:power}).  Let us recall that these spaces play a decisive role for Sobolev embeddings into spaces with slowly decaying Frostman measures, as was recently pointed out in~\cite[Theorem~5.1]{Cia:20}.

Finally, given two rearrangement-invariant function spaces $X(\RR, \mu)$ and $Y(\RR, \mu)$ over the same measure space, their sum $(X + Y)(\RR, \mu)$ endowed with
\begin{equation*}
    \|f\|_{(X + Y)(\RR, \mu)} = K(f, 1; X, Y),\ f\in\M^+(\RR, \mu),
\end{equation*}
is also a rearrangement-invariant function space. Here $K$ is the Peetre $K$-functional defined as, for $f\in\M^+(\RR, \mu)$ and $t\in(0,\infty)$,
\begin{equation*}
    K(f, t; X, Y) = \inf_{f=g+h} \Big( \|g\|_{X(\RR, \mu)} + t \|h\|_{Y(\RR, \mu)} \Big).
\end{equation*}
The $K$-functional is nondecreasing in $t$ and the function $(0, \infty) \ni t \mapsto t^{-1}K(f, t; X, Y)$ is nonincreasing.
We have, for all $a,b>0$,
\begin{equation}\label{prel:K_equiv_norms}
  \min\Big\{ \frac{a}{b},1 \Big\}K(f, b; X, Y) \leq K(f, a; X, Y) \leq \max\Big\{ \frac{a}{b},1 \Big\}K(f, b; X, Y).
\end{equation}
Equivalent expressions for the $K$-functional between a pair of function spaces are known for a large number of function spaces. For example, see~\cite{H:70} for the expression of the $K$-functional for a pair of Lorentz spaces (in particular, for a pair of Lebesgue spaces). The interested reader can find more information about the $K$-functional in \cite[Chapter~5]{BS}.

Finally, every rearrangement-invariant function space $X(\RR, \mu)$ is contained in $(L^1 + L^\infty)(\RR, \mu)$ (e.g., see~\cite[Chapter~2, Theorem~6.6]{BS}).

\section{Sawyerability and properties of the governing operator}\label{sec:calderon_general_theory}
We start with a characterization of sawyerable operators.
\begin{proposition}\label{prop:joint_sawyerable_vs_sep_sawyerable}
Let $(\RR,\mu)$ and $(\cS, \nu)$ be nonatomic $\sigma$-finite measure spaces. Assume that $p,\pone$ satisfy \eqref{E:pq}, and let $r$ be defined by~\eqref{def:sawyerable_op:def_of_p}. Let $T$ be a quasi-linear operator defined on $(L^{p} + L^{\pone , \infty})(\RR, \mu)$ and taking values in $\M_0(\cS, \nu)$. Then the following three statements are equivalent.
\begin{enumerate}[(i)]
    \item The operator $T$ is $(p,\pone)$-sawyerable, i.e., it satisfies the endpoint estimates~\eqref{saw_op_p_p_endpoint} and \eqref{saw_op_p1_infty_endpoint}.
    \item There is a constant $C>0$ such that
\begin{equation}\label{prop:joint_sawyerable_vs_sep_sawyerable:pointwise}
    \big( |Tf|^p \big)_\nu^{**} (t)^\frac1{p} \leq C 
    \Big( R_{p,\pone}f_\mu^*(t) + \sup_{s\in[t^{\ptwo }, \infty)} s^\frac1{\pone } f_\mu^*(s) \Big) \quad \text{for all $t\in(0, \nu(\cS))$}
\end{equation}
and every $f\in (L^{p} + L^{\pone , \infty})(\RR, \mu)$,
where the operator $R_{p,\pone}$ is defined by~\eqref{def:sawyerable_op:def_of_R}.
\item There is a positive constant $C>0$ such that
\begin{equation}\label{prop:joint_sawyerable_vs_sep_sawyerable:K}
        K(Tf, t; L^p, L^\infty) \leq C K(f, t; L^p, L^{\pone , \infty}) \quad \text{for all $t\in(0,\infty)$}
\end{equation}
and every $f\in (L^{p} + L^{\pone , \infty})(\RR, \mu)$.
\end{enumerate}
\end{proposition}
\begin{proof}
First, assume that (i) is true, i.e., $T$ is bounded from $L^p(\RR,\mu)$ and $L^{\pone , \infty}(\RR,\mu)$ to $L^p(\cS,\nu)$ and $L^\infty(\cS,\nu)$, respectively. Since $T$ is quasi-linear, it follows (see~\cite[Proposition~3.1.15]{BK:91} and \cite[Chapter~5, Theorem~1.11]{BS}) that \eqref{prop:joint_sawyerable_vs_sep_sawyerable:K} is true with a constant $C>0$ depending only on $\|T\|_{L^p\to L^p}$, $\|T\|_{L^{\pone ,\infty} \to L^\infty}$, and the constant $k$ from the definition of the quasi-linearity. In other words, (i) implies (iii).

Next, we prove the reverse implication. Assume that~\eqref{prop:joint_sawyerable_vs_sep_sawyerable:K} is valid. Using the well-known equivalent expression for the $K$-functional between $L^p$ and $L^\infty$ (e.g., see~\cite[Theorem~5.2.1]{BL}), we have
    \begin{equation}\label{prop:joint_sawyerable_vs_sep_sawyerable:eq1}
        \Bigg( \int_0^{t^p} (Tf)_\nu^*(s)^p \ds \Bigg)^\frac1{p} \lesssim K(f, t; L^p, L^{\pone , \infty}) \quad \text{for every $t\in(0,\infty)$}
    \end{equation}
    and every $f\in (L^{p} + L^{\pone , \infty})(\RR, \mu)$. Here the multiplicative constant depends only on $C$ from \eqref{prop:joint_sawyerable_vs_sep_sawyerable:K}, $p$ and $\pone$. The trivial decomposition $f=f+0$ shows that
    \begin{equation*}
        K(f, t; L^p, L^{\pone , \infty})
        \le 
        \|f\|_{L^{p}(\RR, \mu)}
        \quad\text{for every $t\in(0,\infty)$.}
    \end{equation*}
    Thus, letting $t\to\infty$ in \eqref{prop:joint_sawyerable_vs_sep_sawyerable:eq1}, we obtain
    \begin{align*}
        \|Tf\|_{L^p(\cS,\nu)} &
        \lesssim 
        \|f\|_{L^{p}(\RR, \mu)}
    \end{align*}
    for every $f\in L^p(\RR,\mu)$. In other words, $T\colon L^p(\RR,\mu) \to L^p(\cS,\nu)$. Similarly, the decomposition $f=0+f$ leads to
    \begin{equation*}
        K(f, t; L^p, L^{\pone , \infty})
        \le 
        t\|f\|_{L^{\pone , \infty}(\RR, \mu)}
        \quad\text{for every $t\in(0,\infty)$.}
    \end{equation*}
    Consequently, dividing \eqref{prop:joint_sawyerable_vs_sep_sawyerable:eq1} by $t$ and letting $t\to 0^+$, we obtain
    \begin{align*}
        \|Tf\|_{L^\infty(\cS,\nu)} 
        &
        \lesssim 
        \|f\|_{L^{\pone , \infty}(\RR, \mu)}
    \end{align*}
    for every $f\in L^{\pone , \infty}(\RR,\mu)$. In other words, $T\colon L^{\pone , \infty}(\RR,\mu) \to L^\infty(\cS,\nu)$. Altogether, we have shown that (iii) implies (i).

Finally, we show that the statements (ii) and (iii) are equivalent, which will finish the proof. Observing that \eqref{def:sawyerable_op:def_of_p} can be expressed as $\frac1{p} - \frac1{\pone } = \frac1{p \ptwo }$, and using the well-known equivalent expressions for the $K$-functionals between Lorentz/Lebesgue spaces (see~\cite[Theorem~4.2]{H:70}), it is easy to see that \eqref{prop:joint_sawyerable_vs_sep_sawyerable:K} is valid if and only if
\begin{equation}
\label{prop:joint_sawyerable_vs_sep_sawyerable:K+Holm}
        \Bigg( \int_0^{t^p} (Tf)_\nu^*(s)^p \ds \Bigg)^\frac1{p} \lesssim \Bigg( \int_0^{t^{p \ptwo }} f_\mu^*(s)^p \ds \Bigg)^\frac1{p} + t\sup_{s\in[t^{p \ptwo }, \infty)} s^\frac1{\pone } f_\mu^*(s)
    \end{equation}
    for every $t\in(0, \infty)$ and every $f\in (L^{p} + L^{\pone , \infty})(\RR, \mu)$. The multiplicative constant in \eqref{prop:joint_sawyerable_vs_sep_sawyerable:K+Holm} depends only on that in \eqref{prop:joint_sawyerable_vs_sep_sawyerable:K}, $p$ and $\pone$. Multiplying \eqref{prop:joint_sawyerable_vs_sep_sawyerable:K+Holm} by $t^{-1}$ and using the definition of $(Tf)_\nu^{**}$ and of $R_{p,\pone}f^*_\mu$,
    we obtain
    \begin{equation}
    \label{prop:joint_sawyerable_vs_sep_sawyerable:eq2}
        \big( |Tf|^p \big)_\nu^{**} (t^p)^\frac1{p} \lesssim R_{p,\pone}f^*_\mu(t^p) + \sup_{s\in[t^{p \ptwo }, \infty)} s^\frac1{\pone } f_\mu^*(s)
    \end{equation}
    for every $t\in(0, \infty)$ and every $f\in (L^{p} + L^{\pone , \infty})(\RR, \mu)$.
    The simple change of variables $t^p\mapsto t$ shows that~\eqref{prop:joint_sawyerable_vs_sep_sawyerable:eq2} is equivalent to~\eqref{prop:joint_sawyerable_vs_sep_sawyerable:pointwise},
    and establishes thereby the implication (iii)$\Rightarrow$(ii).
    
    In order to prove the converse implication, assume that \eqref{prop:joint_sawyerable_vs_sep_sawyerable:pointwise} is valid. Then the above analysis shows that \eqref{prop:joint_sawyerable_vs_sep_sawyerable:eq2}, and hence also~\eqref{prop:joint_sawyerable_vs_sep_sawyerable:K+Holm}, holds for every $f\in (L^{p} + L^{\pone , \infty})(\RR, \mu)$ and every $t\in(0, \nu(\cS))$. Thus, if $\nu(\cS) = \infty$, then \eqref{prop:joint_sawyerable_vs_sep_sawyerable:K} immediately follows. When $\nu(\cS) < \infty$, \eqref{prop:joint_sawyerable_vs_sep_sawyerable:pointwise}~only implies that \eqref{prop:joint_sawyerable_vs_sep_sawyerable:eq2} is true for every $t\in(0, \nu(\cS)^{\frac1{p}}]$, and, consequently, so is \eqref{prop:joint_sawyerable_vs_sep_sawyerable:K}. However, since $(L^p + L^\infty)(\cS, \nu) = L^p(\cS, \nu)$ provided that $\nu(\cS) < \infty$, it is not hard to see that
    	\begin{align*}
			K(Tf, t; L^p, L^\infty) &\approx K(Tf, \nu(\cS)^{\frac1{p}}; L^p, L^\infty) \lesssim K(f, \nu(\cS)^{\frac1{p}}; L^p, L^{\pone , \infty}) 
                \\
			& \leq K(f, t; L^p, L^{\pone , \infty})
		\end{align*}
	for every $t > \nu(\cS)^{\frac1{p}}$ and every $f\in (L^{p} + L^{\pone , \infty})(\RR, \mu)$. This, once again, establishes the validity of~\eqref{prop:joint_sawyerable_vs_sep_sawyerable:K}. Hence, putting everything together, we see that the statements (ii) and (iii) are indeed equivalent. The proof is complete.
\end{proof}

\begin{remark}\ 
\begin{enumerate}[(i)]
\item The statement that $f\in (L^{p} + L^{\pone , \infty})(\RR, \mu)$ is equivalent to the condition $R_{p,\pone}f^*_\mu(1) + \sup_{s\in[1, \infty)} s^{\frac1{\pone }}f_\mu^*(s)<\infty$. Furthermore, when $\mu(\RR) < \infty$, the sum $(L^{p} + L^{\pone , \infty})(\RR, \mu)$ coincides with $L^p(\RR, \mu)$, up to equivalent norms, and $f\in (L^{p} + L^{\pone , \infty})(\RR, \mu) = L^p(\RR, \mu)$ is equivalent to $R_{p,\pone}f^*_\mu(1)<\infty$.
\item We could replace the pointwise inequality~\eqref{prop:joint_sawyerable_vs_sep_sawyerable:pointwise} with a seemingly more general inequality:
\begin{equation*}
     \big( |Tf|^p \big)_\nu^{**} (ct)^\frac1{p} \leq C 
    \Big( R_{p,\pone}f_\mu^*(t) + \sup_{s\in[t^{\ptwo }, \infty)} s^\frac1{\pone } f_\mu^*(s) \Big) \quad \text{for every $t\in(0, \nu(\cS))$},
\end{equation*}
where $c>0$ is another constant independent of both $f$ and $t$. However, if this holds with some $c>0$, so it does with $c = 1$ and a possibly different constant $C>0$. This follows from the observation that
\begin{equation*}
    \big( |Tf|^p \big)_\nu^{**} (ct)\geq \min\{1, c^{-1}\}\big( |Tf|^p \big)_\nu^{**} (t) \quad \text{for every $t\in(0,\infty)$}.
\end{equation*}
Therefore, the choice of $c=1$ is without any loss of generality.
\end{enumerate}
\end{remark}

The following proposition tells us that the supremum operator in the pointwise estimate \eqref{prop:joint_sawyerable_vs_sep_sawyerable:pointwise} is in fact essentially immaterial for rearrangement-invariant norm inequalities. In other words, it basically explains why sawyerable operators are governed only by the operator $R_{p,\pone}$. Interestingly, a similar phenomenon was observed in connection with a class of operators with completely different nonstandard endpoint behavior in~\cite[Theorem~1.2]{GP-Indiana:09}.

\begin{proposition}\label{pro:supremum_operator_is_unnecessary}
    Assume that $p,\pone$ satisfy~\eqref{E:pq} and let $r$ be defined by~\eqref{def:sawyerable_op:def_of_p}. There is a constant $C>0$ depending only on $p$ and $\pone$ such that
    \begin{equation*}
        \left\| \sup_{s\in[t^{\ptwo }, a^{\ptwo })} s^\frac1{\pone } h^*(s) \right\|_{Z(0, a)} \leq C \|R_{p,\pone}h\|_{Z(0, a)} \quad \text{for every $h\in\M(0, \infty)$},
    \end{equation*}
		for every $a\in(0, \infty]$, and for every rearrangement-invariant function space $Z(0, a)$.
\end{proposition}
\begin{proof}
Fix $h\in\M(0, \infty)$. Notice that
\begin{equation*}
R_{p,\pone} h (t) = \big( t^{\ptwo  - 1} (|h|^p)^{**}(t^{\ptwo }) \big)^\frac1{p} \quad \text{for every $t\in(0, \infty)$},
\end{equation*}
and
\begin{equation}\label{pro:supremum_operator_is_unnecessary:eq1}
     \big( t^{\ptwo  - 1} (|h|^p)^{**}(t^{\ptwo }) \big)^\frac1{p}
     \geq \big( t^{\ptwo  - 1} h^*(t^{\ptwo })^p \big)^\frac1{p}
     =
     t^{\frac{\ptwo }{\pone }}h^*(t^{\ptwo }) \quad \text{for every $t\in(0,\infty)$}
\end{equation}
thanks to the relation $(|h|^p)^{**}\geq (|h|^p)^*=(h^*)^p$ and \eqref{def:sawyerable_op:def_of_p}. Owing to~\cite[Lemma~3.1(ii)]{GP-Indiana:09} (with $\beta = 0$ and $\alpha=\frac{\ptwo}{\pone} $ in their notation), we have
\begin{equation*}
    \int_0^t \sup_{\tau\in[s^{\ptwo }, a^{\ptwo })} \tau^\frac1{\pone } h^*(\tau) \ds \lesssim \int_0^t \big(\chi_{(0,a)}(\tau)\tau^\frac{\ptwo }{\pone } h^*(\tau^{\ptwo }) \big)^*(s) \ds \quad \text{for every $t\in(0,\infty)$},
\end{equation*}
in which the multiplicative constant depends only on $p$ and $\pone$. Combining this with \eqref{pro:supremum_operator_is_unnecessary:eq1}, we obtain
\begin{equation*}
	\int_0^t \sup_{\tau\in[s^{\ptwo }, a^{\ptwo })} \tau^\frac1{\pone } h^*(\tau) \ds \lesssim \int_0^t \big( \chi_{(0,a)} R_{p,\pone}h \big)^*(s) \ds \quad \text{for every $t\in(0,\infty)$}.
\end{equation*}
Hence, it follows from the so-called Hardy--Littlewood--P\'{o}lya principle (see~\cite[Chapter~2, Theorem~4.6]{BS}) and the monotonicity of the function
\begin{equation*}
    (0, \infty)\ni t\mapsto \chi_{(0,a)}(t)\sup_{s\in[t^{\ptwo }, a^{\ptwo })} s^\frac1{\pone } h^*(s)
\end{equation*}
that
\begin{equation*}
\left\| \sup_{s\in[t^{\ptwo }, a^{\ptwo })} s^\frac1{\pone } h^*(s) \right\|_{Z(0, a)} \lesssim \| R_{p,\pone}h\|_{Z(0, a)}.\qedhere
\end{equation*}
\end{proof}

It is important to notice that the operator $R_{p,\pone}$ itself is not necessarily $(p,\pone)$-sawyerable. More precisely, it satisfies an essentially weaker endpoint estimate than \eqref{saw_op_p_p_endpoint}. This fact is the content of our next proposition. At the same time, this is precisely the stage of our analysis at which the $(\cdot)^{\langle p\rangle}$-operation on function spaces comes into play, improving the boundedness properties of $(p,\pone)$-sawyerable operators.
\begin{proposition}\label{prop:endpoint_estimates_for_R}
Assume that $p,\pone$ satisfy~\eqref{E:pq} and let $r$ be defined by~\eqref{def:sawyerable_op:def_of_p}. The operator $R_{p,\pone}$ defined by \eqref{def:sawyerable_op:def_of_R} is sublinear, and it is bounded from $L^p(0, \infty)$ and $L^{\pone , \infty}(0, \infty)$ to $L^{p, \infty}(0, \infty)$ and $L^\infty(0, \infty)$, respectively.
\end{proposition}
\begin{proof}
    We clearly have
    \begin{equation*}
       |R_{p,\pone}(\alpha f)|  = R_{p,\pone}(\alpha f) = |\alpha| R_{p,\pone}f = |\alpha| |R_{p,\pone}f|
    \end{equation*}
    for every $f\in\M(0, \infty)$ and every $\alpha\in\R$ thanks to the positive homogeneity of the operation $f\mapsto f^{**}$. As for the subadditivity, using \eqref{prel:double_star_alt_expr}, we observe that
    \begin{align*}
        R_{p,\pone}(f+g)(t) &= t^{-\frac1{p}} \sup_{\substack{|E| = t^{\ptwo }\\E\subseteq(0, \infty)}} \left(\int_E |f + g|^p \right)^{\frac1{p}} \\
        &\leq t^{-\frac1{p}}\sup_{\substack{|E| = t^{\ptwo }\\E\subseteq(0, \infty)}} \left(\int_E |f|^p \right)^{\frac1{p}} + t^{-\frac1{p}}\sup_{\substack{|E| = t^{\ptwo }\\E\subseteq(0, \infty)}} \left(\int_E |g|^p \right)^{\frac1{p}} \\
        &= R_{p,\pone}f(t) + R_{p,\pone}g(t)
    \end{align*}
    for every $t\in(0,\infty)$ and all $f,g\in\M_0(0, \infty)$. Next, note that
    \begin{align*}
        \|R_{p,\pone}f\|_{L^{p, \infty}(0, \infty)}^p 
        &= \sup_{t\in(0, \infty)}t (R_{p,\pone}f)^*(t)^p
            \leq \sup_{t\in(0, \infty)} t \sup_{s\in[t, \infty)} s^{\ptwo -1} (|f|^p)^{**}(s^{\ptwo }) 
            \\
        &\leq \sup_{t\in(0, \infty)} \sup_{s\in[t, \infty)} s^{\ptwo } (|f|^p)^{**}(s^{\ptwo }) 
            = \int_0^\infty (f^*)^p = \|f\|_{L^p(0, \infty)}^p
    \end{align*}
    for every $f\in L^p(0, \infty)$. Finally, using~\eqref{E:pq}, \eqref{def:sawyerable_op:def_of_p} and \eqref{prel:two_stars_on_Lorentz_spaces}, we obtain
    \begin{align*}
        \|R_{p,\pone}f\|_{L^{\infty}(0, \infty)}^p 
        &= 
        \sup_{t\in(0, \infty)} t^{\ptwo  - 1} (|f|^p)^{**}(t^{\ptwo }) = \sup_{t\in(0, \infty)} t^{\frac1{\ptwo '}} (|f|^p)^{**}(t) 
        \\
        &= \|(|f|^p)^{**}\|_{L^{\ptwo ', \infty}(0, \infty)} \lesssim \||f|^p\|_{L^{\ptwo ', \infty}(0, \infty)} 
            \\
        &= \|f\|_{L^{p\ptwo ', \infty}(0, \infty)}^p = \|f\|_{L^{\pone , \infty}(0, \infty)}^p
    \end{align*}
    for every $f\in L^{\pone , \infty}(0, \infty)$.
\end{proof}

We now turn our attention to specializing Theorems~\ref{thm:calderon_thm_for_sawyerable_op} and \ref{T:general-principle}. We start with customizing Theorem~\ref{T:general-principle} to the generalized potential $I_\alpha^\mu$ defined by~\eqref{generalized_potential_def}, which will serve as the ``bad operator'' $B$ in \eqref{E:pointwise-general}. For this ``bad operator'', the ``good operator'' $G$ is a suitable fractional maximal function (or rather, a suitable family of fractional maximal functions).

For a cube $Q_0\subseteq \mathbb{R}^n$, let $\mathcal{D}(Q_0)=\{ 2^{-k} (\mathbf{n}+Q_0), k \in \mathbb{Z}, \mathbf{n} \in l(Q_0)R(\mathbb{Z}^n)\}$ be the dyadic lattice generated by $Q_0$ (where $R \in SO(n)$ is a rotation which takes the standard basis of $\mathbb{Z}^n$ to a canonical basis one can associate with $Q_0$), and $\mathcal{D}_\tau(Q_0)$ denote the set of $3^n$ translates of this lattice by~$1/3$.  For $\beta \in [0,d)$, $d\in(0,n]$, we define the dyadic fractional maximal function
\begin{align}\label{generalized_maximal_def}
\mathcal{M}^{\mu,Q_0,\tau}_{\beta} f(x)= \sup_{Q \in \mathcal{D}_\tau(Q_0)} \chi_Q(x) \mu(Q)^{\beta/d-1}\int_Q |f(y)|\intdif{\mu(y)},\ x\in\rn,
\end{align}
where $\mu$ is a Radon measure on $\rn$ satisfying \eqref{E:mainresult-prime:d_upper_ahlfors}.

The following lemma connects $I_\alpha^\mu$ and $\mathcal{M}^{\mu,Q_0,\tau}_{\beta}$ in the spirit of \eqref{E:pointwise-general} (cf.~\cite[Proposition~3.1.2]{AH}). What is also important to note here is that nonfractional maximal functions are usually uniformly pointwise bounded over bounded subsets of $L^\infty$. In particular, this is the case for $\mathcal{M}^{\mu,Q_0,\tau}_{0}$ and the set of characteristic functions on $\rn$. 
\begin{lemma}\label{bigger_lemma}
Let $Q_0 \subseteq \mathbb{R}^n$ be a cube and suppose $\mu$ is a Radon measure on $\rn$ satisfying \eqref{E:mainresult-prime:d_upper_ahlfors}. For $\beta \in (0,d)$ and $\theta \in (0,1)$, there exists a constant $C=C(\beta,\theta,n)>0$ such that 
\begin{align*}
|I^\mu_{\theta \beta} f(x) | \leq C \mathcal{M}^{\mu,Q_0,\tau}_{\beta} f(x)^\theta \mathcal{M}^{\mu,Q_0,\tau}_{0}f(x)^{1-\theta} \quad \text{for every $x\in\rn$}.
\end{align*}
\end{lemma}

\begin{proof}
For any $r>0$ one has
\begin{align*}
|I_{\theta \beta}^\mu f(x)| &\leq \int_{B(x,r)} \frac{|f(y)|}{|x-y|^{d-\theta \beta}}\intdif{\mu(y)} + \int_{\rn\setminus B(x,r)} \frac{|f(y)|}{|x-y|^{d-\theta \beta}}\intdif{\mu(y)} \\
&=: I+II.
\end{align*}
Dyadic annular expansion on $I$ yields
\begin{align*}
I &\leq \sum_{j=0}^\infty (2^{-j-1}r)^{\theta \beta-d} \int_{B(x,2^{-j}r)\setminus B(x,2^{-j-1}r)}|f(y)| \intdif{\mu(y)} \\
&\leq \sum_{j=0}^\infty (2^{-j-1}r)^{\theta \beta-d} \int_{B(x,2^{-j}r)}|f(y)| \intdif{\mu(y)}.
\end{align*}
One then uses the $1/3$ trick, that $B(x,2^{-j}r) \subseteq 3Q$ for some cube $Q \in \mathcal{D}_\tau(Q_0)$ with $l(Q) \approx 2^{-j}r$, two sided comparable.  For each $j$ we set $Q_j=3Q$, which yields
\begin{align*}
I &\leq \sum_{j=0}^\infty (2^{-j-1}r)^{\theta \beta-d} \int_{Q_j}|f(y)| \intdif{\mu(y)} \\
&= \sum_{j=0}^\infty (2^{-j-1}r)^{\theta \beta-d} \frac{\mu(Q_j)}{\mu(Q_j)} \int_{Q_j}|f(y)| \intdif{\mu(y)} \\
 &\lesssim \sum_{j=0}^\infty (2^{-j-1}r)^{\theta \beta-d} 
 (2^{-j}r)^d \mathcal{M}^{\mu,Q_0,\tau}_{0}f(x)\\
 &\leq Cr^{\theta \beta} \mathcal{M}^{\mu,Q_0,\tau}_{0}f(x).
\end{align*}
A similar argument applies to $II$:
\begin{align*}
II&\leq \sum_{j=0}^\infty (2^{j}r)^{\theta \beta-d} \int_{B(x,2^{j+1}r)\setminus B(x,2^{j}r)}|f(y)| \intdif{\mu(y)} \\
&\leq \sum_{j=0}^\infty (2^{j}r)^{\theta \beta-d} \int_{B(x,2^{j+1}r)}|f(y)| \intdif{\mu(y)}\\
&\leq \sum_{j=0}^\infty (2^{j}r)^{\theta \beta-d}  \frac{\mu(Q_j)^{1-\beta/d}}{\mu(Q_j)^{1-\beta/d}}\int_{Q_j}|f(y)| \intdif{\mu(y)}\\
 &\leq Cr^{\theta \beta-\beta} \mathcal{M}^{\mu,Q_0,\tau}_{\beta}f(x).
\end{align*}
While the desired inequality now follows from optimization, e.g. the choice
\begin{equation*}
r^{- \beta} = \frac{\mathcal{M}^{\mu,Q_0,\tau}_{0}f(x)}{ \mathcal{M}^{\mu,Q_0,\tau}_{\beta}f(x)}.\qedhere
\end{equation*}
\end{proof}

\begin{remark}
Note that in the proof above, one uses the structure of $\mathbb{R}^n$ to find canonical dyadic cubes which contain any ball, after which one only needs the polynomial bound on the growth of the measure $\mu$, $\mu(Q) \lesssim l(Q)^d$.
\end{remark}

Equipped with the lemma, we are in a position to prove a general boundedness result for $I_\alpha^\mu$, of which Theorem~\ref{mainresult-prime} is a special case.
\begin{theorem}\label{thm:most_general_KK}
Let $0<d \leq n$, $\alpha \in (0,d)$, and $1<p<\frac{d}{\alpha}$. Let $\mu$ be a Radon measure on $\rn$ such that
\begin{equation*}
\sup_{Q} \frac{\mu(Q)}{l(Q)^{d}} < \infty.
\end{equation*}
For a cube $Q_0\subseteq \rn$, let $\{\mathcal D_j\}_{j = 1}^{3^n}$ be an enumeration of $\mathcal D_\tau(Q_0)$, the previously defined set of $3^n$ translates of $\mathcal{D}(Q_0)$ by $\frac1{3}$. Assume that $\nu$ is a Radon measure on $\rn$ such that
\begin{equation*}
       \sup_{\mu(Q)>0} \frac{\nu(Q)}{\mu(Q)^{1-\frac{\alpha p}{d}}} < \infty,
\end{equation*}
where the supremum extends over all $Q\in\bigcup_{j = 1}^{3^n} \mathcal D_j$ with $\mu(Q)>0$.

If $X(\rn, \mu)\subseteq (L^{p} + L^{\frac{d}{\alpha}, \infty})(\rn,\mu)$ and $Y(\rn, \nu)$ are rearrangement-invariant function spaces such that \eqref{T:general-principle:fundX_vanishes} is satisfied and that $R_{p,\frac{d}{\alpha}}\colon \bar X(0, \mu(\rn)) \to \bar Y(0, \nu(\rn))$, then
\begin{equation}\label{thm:most_general_KK:potential_bdd}
        I_\alpha^\mu\colon \Lambda_X(\rn, \mu) \to Y^{\langle p \rangle}(\rn, \nu).
\end{equation}
\end{theorem}
\begin{proof}
Let us recall that $I_\alpha^\mu$ is the generalized Riesz potential defined by \eqref{generalized_potential_def}.  The fact that $p>1$ allows us to find $\delta \in (\alpha, \alpha p)$.  For this fixed $\delta>\alpha$, we define
\begin{equation*}
        Gf = \sum_{j=1}^{3^n} \mathcal{M}_{\delta}^{\mu, j} f,\ f\in\M(\rn, \mu),
\end{equation*}
where $\mathcal{M}_{\delta}^{\mu,j}$ is the weighted fractional maximal operator corresponding to the dyadic grid $\mathcal D_j$ defined as
    \begin{equation*}
        \mathcal{M}_{\delta}^{\mu,j} f(x) = \sup_{Q\in \mathcal D_j} \chi_Q(x) \mu(Q)^{\frac{\delta}{d} - 1} \int_Q |f| \intdif{\mu},\ f\in\M(\rn, \mu).
    \end{equation*}
		
Next, it is easy to see that the operator $G$ is sublinear and that
\begin{equation}\label{rem:generalized_KK:p1_bound}
        G\colon L^{\frac{d}{\delta}, \infty}(\rn, \mu) \to L^\infty(\rn, \mathcal H^0).
\end{equation}
Here $\mathcal H^0$ is the counting measure on $\rn$ and $L^\infty(\rn,\mathcal H^0)$ is the space of everywhere bounded functions. Furthermore, it is not hard to see that our assumptions on the measures $\mu$ and $\nu$ imply that
    \begin{equation*}
     \nu(Q) \mu(Q)^{\frac{\alpha p}{d}}  \lesssim \mu(Q)\quad \text{for every $Q\in \mathcal D_j$, $j=1, \dots, 3^n$}.
    \end{equation*}
    Hence, for every $j = 1, \dots, 3^n$,
    \begin{equation*}
        \mathcal{M}_{\delta}^{\mu,j}\colon L^{\frac{\alpha p}{\delta}}(\rn, \mu) \to L^{\frac{ \alpha p}{\delta}}(\rn, \nu)
    \end{equation*}    
    thanks to~\cite[Theorem~A]{Sawyer1}. It follows that    \begin{equation}\label{rem:generalized_KK:p_bound}
        G\colon L^{\frac{\alpha p}{\delta}}(\rn, \mu) \to L^{\frac{\alpha p}{\delta}}(\rn, \nu).
    \end{equation}
    In view of \eqref{rem:generalized_KK:p1_bound} and \eqref{rem:generalized_KK:p_bound}, we see that the operator $G$ is $(\frac{\alpha p}{\delta},\frac{d}{\delta})$-sawyerable, while we also note that
$G$ is order preserving.	
Moreover, notice that
\begin{equation*}
\mathcal{M}^{\mu,Q_0,\tau}_{\delta} f\leq Gf \quad \text{for every $f\in\M(\rn, \mu)$},
\end{equation*}
where $\mathcal{M}^{\mu,Q_0,\tau}_{\delta}$ is the maximal operator defined by \eqref{generalized_maximal_def}. Furthermore, one has
\begin{equation*}
\mathcal{M}^{\mu,Q_0,\tau}_{0}\chi_E\leq 1
\end{equation*}
for every $\mu$-measurable $E\subseteq\rn$. Therefore, it follows from Lemma~\ref{bigger_lemma} with $\theta = \frac{\alpha}{\delta}$ and $\beta = \delta$ (note that $\theta\in(0,1)$ and $\beta\in(0,d)$)
that
\begin{equation*}
        |I_\alpha^\mu \chi_E(x)| \leq C(\alpha, \delta, d,n) G\chi_E(x)^\frac{\alpha}{\delta} \quad \text{for every $x\in\rn$}.
\end{equation*}
Therefore, \eqref{E:pointwise-general} with $B=I_{\alpha}^\mu$ and $\lambda = \frac{\alpha}{\delta}$ is true with the same multiplicative constant. It remains to observe that, owing to our choice of $\delta$, the condition \eqref{E:parameters-large-enough} is satisfied. Therefore, altogether, we obtain \eqref{thm:most_general_KK:potential_bdd} by virtue of Theorem~\ref{T:general-principle}.
\end{proof}

Any effective use of Theorem~\ref{thm:calderon_thm_for_sawyerable_op} in practical tasks would require knowledge of boundedness of sawyerable operators on customary function spaces. The following result specializes it to Lorentz spaces (and so also to Lebesgue spaces).

\begin{theorem}
Let $\mu(\RR)=\infty$ and $r_1,r_2,s_1,s_2\in[1,\infty]$. 
Assume that $p,\pone$ satisfy~\eqref{E:pq}, and let $r$ be defined by~\eqref{def:sawyerable_op:def_of_p}. Suppose that either
    \begin{align}
        r_1\in(p,\pone ),\quad s_1\le s_2,\quad\text{and}\quad\frac{1}{\pone }+\frac{1}{\ptwo r_2}=\frac{1}{r_1}\label{E:sawyer-lorentz-condition-i-1},
            \\
        \intertext{or}
        r_1=r_2=p\quad\text{and}\quad s_1\leq p\leq s_2 \label{E:sawyer-lorentz-condition-i-2},
            \\
        \intertext{or}
        r_1=\pone ,\quad r_2=\infty\quad\text{and}\quad s_2=\infty.\label{E:sawyer-lorentz-condition-i-3}
    \end{align}
    Then every $(p,\pone)$-sawyerable operator $T$ is bounded from $L^{r_1,s_1}(\RR,\mu)$ to $L^{r_2,s_2}(\cS,\nu)$.
    
    Furthermore, when $\nu(\cS)<\infty$, we may replace $L^{r_2,s_2}(\cS,\nu)$ with $L^{r,s}(\cS,\nu)$ for every $1\leq r<r_2$ and $s\in[1,\infty]$.
\end{theorem}

\begin{proof}
First, assume that~\eqref{E:sawyer-lorentz-condition-i-1}  holds. Note that it coincides with~\eqref{E:R-lorentz-parameters-i}. Hence
   \begin{equation*}
       R_{p,\pone}\colon L^{r_1,s_1}(0, \infty)\to L^{r_2,s_2}(0, \nu(\cS))
   \end{equation*}
   by Theorem~\ref{thm:R_boundedness_lorentz_infty}. Since $r_2>r_1>p>1$, one has
   \begin{equation*}
       (L^{r_2,s_2})^{\langle p\rangle}(\cS, \nu) = L^{r_2,s_2}(\cS, \nu)
   \end{equation*}
   owing to \eqref{ex:sawyerable_T_boundedness_lorentz_infty:power}. Therefore, the claim follows from Theorem~\ref{thm:calderon_thm_for_sawyerable_op}.

If either \eqref{E:sawyer-lorentz-condition-i-2} or~\eqref{E:sawyer-lorentz-condition-i-3} is satisfied, then the assertion follows straightaway from the definition of a sawyerable operator combined with the nesting property of Lorentz spaces pointed out in~\eqref{prel:lorentz_second_index_increasing}.

Finally, assume that $\nu(\cS)<\infty$ and either \eqref{E:sawyer-lorentz-condition-i-1} or \eqref{E:sawyer-lorentz-condition-i-2} is valid. We have already proved that every $(p,\pone)$-sawyerable operator $T$ is bounded from $L^{r_1,s_1}(\RR,\mu)$ to $L^{r_2,s_2}(\cS,\nu)$. Since $L^{r_2,s_2}(\cS,\nu)\subseteq L^{r,s}(\cS,\nu)$ for every $1\leq r < r_2$ and $s\in[1,\infty]$ provided that  $\nu(\cS)<\infty$ (e.g., \cite[p.~217]{BS}), we immediately obtain the fact that $T\colon L^{r_1,s_1}(\RR,\mu) \to L^{r,s}(\cS,\nu)$ is also bounded. 
\end{proof}

Finally, by combining Theorem~\ref{T:general-principle} with Theorem~\ref{thm:R_boundedness_lorentz_infty}, we obtain the boundedness of ``bad operators'' dominated by suitable ``good ones'' between Lorentz spaces.
\begin{theorem}
    Let $(\RR, \mu)$ and $(\cS, \nu)$ be nonatomic $\sigma$-finite measure spaces. Assume that $p,\pone$ satisfy~\eqref{E:pq} and let $r$ be defined by~\eqref{def:sawyerable_op:def_of_p}.  Let $\lambda\in(0,1)$ satisfy \eqref{E:parameters-large-enough}. Assume that $r_1, r_2, s_1, s_2\in[1,\infty]$ satisfy one of the conditions \eqref{E:sawyer-lorentz-condition-i-1}--\eqref{E:sawyer-lorentz-condition-i-3}. 
    
    Then every linear operator $B$, defined at least on characteristic functions of $\mu$-measurable subsets of $\RR$ of finite measure and taking values in $\M_0(\cS, \nu)$, satisfying \eqref{E:pointwise-general} with some order preserving $(\lambda p,\lambda \pone)$-sawyerable operator $G$, is bounded from $L^{r_1,1}(\RR,\mu)$ to $L^{r_2,s_2}(\cS,\nu)$.
\end{theorem}
\begin{proof}
The claim follows from Theorem~\ref{T:general-principle} combined with Theorem~\ref{thm:R_boundedness_lorentz_infty}, \eqref{ex:sawyerable_T_boundedness_lorentz_infty:power}, and \eqref{prel:Lorentz_spaces_Lambda_endpoint}. When the parameters $r_1,r_2,s_1,s_2$ satisfy either \eqref{E:sawyer-lorentz-condition-i-1} or \eqref{E:sawyer-lorentz-condition-i-3}, we use Theorem~\ref{T:general-principle} with $X(\RR, \mu) = L^{r_1,s_1}(\RR, \mu)$ and $Y(\cS, \nu) = Y^{\langle p \rangle}(\cS, \nu) = L^{r_2,s_2}(\cS, \nu)$. When the parameters satisfy \eqref{E:sawyer-lorentz-condition-i-2}, we use the same theorem, but this time with $X(\RR, \mu) = L^{r_1,s_1}(\RR, \mu)$ and $Y(\cS, \nu) = L^{r_2,\infty}(\cS, \nu)$, recalling that $(L^{p,\infty})^{\langle p \rangle}(\cS,\nu) = L^{p}(\cS, \nu)$.
\end{proof}

   \section{Proofs of Main Results}
\begin{proof}[Proof of Theorem~\ref{thm:calderon_thm_for_sawyerable_op}]
We start by showing that (iii) implies~(ii). Let $T$ be a $(p,\pone)$-sawyerable operator. By Proposition~\ref{prop:joint_sawyerable_vs_sep_sawyerable}, one has
\begin{align*}
    \|Tf\|_{Y^{\langle p \rangle}(\cS,\nu)} &= \| ( (|Tf|^p)^{**}_\nu )^\frac1{p} \|_{\bar Y(0, \nu(\cS))} \\
    &\lesssim \|R_{p,\pone} f^*_\mu\|_{\bar Y(0, \nu(\cS))} + \left \|\sup_{s\in[t^{\ptwo }, \infty)} s^{\frac1{\pone }} f^*_\mu(s)\right\|_{\bar Y(0, \nu(\cS))} \\
    &\approx \|R_{p,\pone} f^*_\mu\|_{\bar Y(0, \nu(\cS))} + \left\|\sup_{s\in[t^{\ptwo }, \nu(\cS)^{\ptwo })} s^{\frac1{\pone }} f^*_\mu(s)\right\|_{\bar Y(0, \nu(\cS))} \\
    &\quad+\left \|\sup_{s\in[\nu(\cS)^{\ptwo }, \infty)} s^{\frac1{\pone }} f^*_\mu(s) \right\|_{\bar Y(0, \nu(\cS))}
\end{align*}
for every $f\in X(\RR, \mu)$. Furthermore, it follows from Proposition~\ref{pro:supremum_operator_is_unnecessary} that
\begin{equation*}
    \left\|R_{p,\pone} f^*_\mu \right\|_{\bar Y(0, \nu(\cS))} + \left\|\sup_{s\in[t^{\ptwo }, \nu(\cS)^{\ptwo })} s^{\frac1{\pone }} f^*_\mu(s)\right\|_{\bar Y(0, \nu(\cS))} \approx \left\|R_{p,\pone} f^*_\mu \right\|_{\bar Y(0, \nu(\cS))}
\end{equation*}
for every $f\in X(\RR, \mu)$. Combining these two observations together with (iii), we arrive at
\begin{align}
    \|Tf\|_{Y^{\langle p \rangle}(\cS,\nu)} &\lesssim \left\|R_{p,\pone} f^*_\mu\right\|_{\bar Y(0, \nu(\cS))} + \left\|\sup_{s\in[\nu(\cS)^{\ptwo }, \infty)} s^{\frac1{\pone }} f^*_\mu(s)\right\|_{\bar Y(0, \nu(\cS))} \nonumber\\
    &\lesssim \|f^*_\mu\|_{\bar X(0,\mu(\RR))} + \left\|\sup_{s\in[\nu(\cS)^{\ptwo }, \infty)} s^{\frac1{\pone }} f^*_\mu(s)\right\|_{\bar Y(0, \nu(\cS))} \nonumber\\
    &= \|f\|_{X(\RR, \mu)} + \left\|\sup_{s\in[\nu(\cS)^{\ptwo }, \infty)} s^{\frac1{\pone }} f^*_\mu(s)\right\|_{\bar Y(0, \nu(\cS))} \label{thm:calderon_thm_for_sawyerable_op:eq9}
\end{align}
for every $f\in X(\RR, \mu)$. Now, since the second term on the right-hand side is equal to zero when $\nu(\cS) = \infty$, we have proved the desired boundedness provided that $\nu(\cS) = \infty$. When $\nu(\cS) < \infty$, we use the fact that $X(\RR,\mu)\subseteq (L^{p} + L^{\pone , \infty})(\RR, \mu)$ and \eqref{prel:K_equiv_norms} to obtain
\begin{equation*}
    K(f,\nu(\cS)^{\frac1{p}}; L^{p}, L^{\pone , \infty}) \approx K(f,1; L^{p}, L^{\pone , \infty}) = \|f\|_{(L^{p} + L^{\pone , \infty})(\RR, \mu)} \lesssim \|f\|_{X(\RR, \mu)}
\end{equation*}
for every $f\in X(\RR, \mu)$, in which the constants in the equivalence depend only on $\nu(\cS)$, $p$ and $q$. Since
\begin{equation*}
    K(f,\nu(\cS)^{\frac1{p}}; L^{p}, L^{\pone , \infty}) \approx \Bigg( \int_0^{\nu(\cS)^{\ptwo }} f_\mu^*(s)^p \ds \Bigg)^\frac1{p} + \nu(\cS)^{\frac1{p}}\sup_{s\in[\nu(\cS)^{\ptwo }, \infty)} s^\frac1{\pone } f_\mu^*(s)
\end{equation*}
for every $f\in X(\RR, \mu)$, thanks to the equivalent expression for the $K$-functional between $L^{p}(\RR, \mu)$ and $L^{\pone , \infty}(\RR, \mu)$ (see~\eqref{prop:joint_sawyerable_vs_sep_sawyerable:K+Holm}), we have
\begin{equation}\label{thm:calderon_thm_for_sawyerable_op:eq10}
    \left\|\sup_{s\in[\nu(\cS)^{\ptwo }, \infty)} s^{\frac1{\pone }} f^*_\mu(s)\right\|_{\bar Y(0, \nu(\cS))} \lesssim \nu(\cS)^{-\frac1{p}} \|f\|_{X(\RR, \mu)} \|1\|_{\bar Y(0, \nu(\cS))}
\end{equation}
for every $f\in X(\RR, \mu)$. Clearly, $\nu(\cS)^{-\frac1{p}}\|1\|_{\bar Y(0, \nu(\cS))} < \infty$ is independent of $f$. Hence, combining \eqref{thm:calderon_thm_for_sawyerable_op:eq9} with \eqref{thm:calderon_thm_for_sawyerable_op:eq10}, we obtain the desired boundedness of $T$ even when $\nu(\cS) < \infty$.

As (ii) clearly implies (i), we only need to prove that (i) implies~(iii). Fix $g=g\chi_{(0, \mu(\RR))}\in \bar X(0, \mu(\RR))$. Replacing $g$ with $g\chi_{(0,N)}\chi_{(0, \mu(\RR))}$ for an appropriate $N\in(0,\infty)$ if necessary, we may assume that $|\spt g| < \infty$, i.e.~the support of $g$ has finite measure. Since $(\RR, \mu)$ is nonatomic, there is a function $h\in X(\RR, \mu) \subseteq (L^{p} + L^{\pone , \infty})(\RR, \mu)$ such that (see~\cite[Chapter~2, Corollary~7.8]{BS})
\begin{equation}\label{thm:calderon_thm_for_sawyerable_op:eq8}
h_\mu^* = g^*.
\end{equation}
Moreover, since $|\spt g| < \infty$, we have $g\in L^{p}(0, \infty)$. Next, by~\cite[Chapter~3, Corollary~2.13]{BS}, there is a linear operator $S\colon (L^1 + L^\infty)(\RR, \mu) \to (L^1 + L^\infty)(0, \infty)$ satisfying:
\begin{align}
    &S|h| = g^* \quad \text{a.e.\ in $(0, \infty)$}, \label{thm:calderon_thm_for_sawyerable_op:eq12}
    \\
    &Sf = \chi_{(0, \mu(\RR))}Sf \quad \text{for every $f\in(L^1 + L^\infty)(\RR, \mu)$}\label{thm:calderon_thm_for_sawyerable_op:eq11}
    \\
    \intertext{and}
    &\max\{\|S\|_{L^1\to L^1}, \|S\|_{L^\infty \to L^\infty}\} \leq1. \label{thm:calderon_thm_for_sawyerable_op:eq1}
\end{align}
Moreover, \eqref{thm:calderon_thm_for_sawyerable_op:eq1} together with \eqref{thm:calderon_thm_for_sawyerable_op:eq11} implies (see~\cite[Chapter~3, Theorem~2.2]{BS}) that
\begin{equation}\label{thm:calderon_thm_for_sawyerable_op:eq2}
    \|Sf\|_{\bar Z(0, \mu(\RR))} \leq \|f\|_{Z(\RR,\mu)} \quad \text{for every $f\in Z(\RR, \mu)$}
\end{equation}
and for every rearrangement-invariant function space $Z(\RR, \mu)$. Now, we define two auxiliary operators. The first one, denoted $Q_1$, is defined as
\begin{equation*}
    Q_1f(t) = \chi_{(0, \nu(\cS))}(t)\Big(\frac1{t^{\ptwo }}\int_0^{t^{\ptwo }} f(s)\chi_{(0, \mu(\RR))}(s) \ds\Big) t^{\frac{\ptwo  - 1}{p}}
\end{equation*}
for  $t\in(0, \infty)$ and $f\in(L^1 + L^\infty)(0, \infty)$.
Clearly
\begin{equation}\label{thm:calderon_thm_for_sawyerable_op:eq3}
   |Q_1f(t)|\leq \chi_{(0, \nu(\cS))}(t)(f\chi_{(0, \mu(\RR))})^{**}(t^{\ptwo })t^{\frac{\ptwo  - 1}{p}} \quad \text{for every $t\in(0, \infty)$}
\end{equation}
and every $f\in(L^1 + L^\infty)(0, \infty)$ thanks to the Hardy--Littlewood inequality~\eqref{prel:HL_ineq}. Next, we set $\tilde{T}= Q_1\circ S$. The operator $\tilde{T}$ is clearly linear, $\tilde{T}f = \chi_{(0, \nu(\cS))}\tilde{T}f$ and $\tilde{T}f\in \M_0(0, \infty)$ for every $f\in (L^1 + L^\infty)(\RR, \mu)$. For future reference, note that
\begin{equation}\label{thm:calderon_thm_for_sawyerable_op:eq6}
\chi_{(0, \nu(\cS))}(t)S|h|(t^{\ptwo }) t^{\frac{\ptwo  - 1}{p}} \leq Q_1(S|h|)(t) = \tilde{T}|h|(t) \quad \text{for every $t\in(0, \infty)$}
\end{equation}
thanks to \eqref{thm:calderon_thm_for_sawyerable_op:eq11} and the fact that $S|h|$ coincides with a nonincreasing function a.e.~in $(0, \infty)$. As for the second auxiliary operator, since $(\cS, \nu)$ is nonatomic and $\tilde{T}|h| = \chi_{(0, \nu(\cS))}\tilde{T}|h|$, there is a function $G\in\M_0^+(\cS,\nu)$ such that $G_\nu^* = (\tilde{T}|h|)^*$ (see~\cite[Chapter~2, Corollary~7.8]{BS}). Moreover, we have
\begin{align}
\lim_{t\to\infty} G_\nu^*(t) &= \lim_{t\to\infty} (\tilde{T}|h|)^*(t) =\lim_{t\to\infty} \big( \chi_{(0, \nu(\cS))}(s)g^{**}(s^{\ptwo }) s^\frac{\ptwo  - 1}{p} \big)^*(t) = 0. \label{thm:calderon_thm_for_sawyerable_op:eq13}
\end{align}
This is obvious when $\nu(\cS)<\infty$. When $\nu(\cS) = \infty$, it is not hard to see that the desired fact follows from $g\in L^{p}(0, \infty)$. Indeed, using the definition of $g^{**}$ and the H\"older inequality, we see that
\begin{equation*}
    g^{**}(s^{\ptwo }) s^\frac{\ptwo  - 1}{p} \leq s^{\frac{\ptwo  - 1}{p} - \ptwo  + \frac{\ptwo }{p'}} \|g\|_{L^p(0,\infty)}  = s^{-\frac1{p}} \|g\|_{L^p(0,\infty)} 
\end{equation*}
for every $s\in(0, \infty)$, and so
\begin{equation*}
    \big(g^{**}(s^{\ptwo }) s^\frac{\ptwo  - 1}{p} \big)^*(t) \leq t^{-\frac1{p}} \|g\|_{L^p(0,\infty)} \quad \text{for every $t\in(0, \infty)$}.
\end{equation*}
Therefore, \eqref{thm:calderon_thm_for_sawyerable_op:eq13} is true regardless of whether $\nu(\cS)<\infty$ or $\nu(\cS)=\infty$. Hence, there is a measure-preserving transformation $\sigma$ from the support of $G$ onto the support of $G_\nu^*$ (i.e., $(0, \nu(\cS))$) such that $G = (G_\nu^* \circ \sigma)\chi_{\spt G}$ (see~\cite[Chapter~2, Corollary~7.6]{BS}). We now define the second auxiliary operator, denoted $Q_2$, as
\begin{equation*}
Q_2 f(x) = f(\sigma(x))\chi_{\spt G}(x),\ x\in \cS,\ f\in\M_0(0, \infty).
\end{equation*}
The operator $Q_2$ is linear and maps $\M_0(0, \infty)$ into $\M_0(\cS, \nu)$. Moreover, we have (see~\cite[Chapter~2, Proposition~7.2]{BS})
\begin{equation}\label{thm:calderon_thm_for_sawyerable_op:eq5}
(Q_2 f)_\nu^* = (\chi_{(0, \nu(\cS))}f)^* \quad \text{a.e.~in $(0,\infty)$ for every $f\in\M_0(0,\infty)$}.
\end{equation}
Hence, in particular,
\begin{equation}\label{thm:calderon_thm_for_sawyerable_op:eq4}
(Q_2(\tilde{T}|h|))_\nu^* = (\tilde{T}|h|)^* \quad \text{a.e.~in $(0,\infty)$}.
\end{equation}

Finally, we define the operator $T$ as
\begin{equation*}
	T = Q_2 \circ \tilde T = Q_2\circ Q_1\circ S.
\end{equation*}
The operator $T$ is clearly linear, being a composition of linear operators. We claim that it is bounded from $L^{p}(\RR,\mu)$ and $L^{\pone , \infty}(\RR, \mu)$ to $L^{p}(\cS, \nu)$ and $L^\infty(\cS, \nu)$, respectively. Indeed, using \eqref{thm:calderon_thm_for_sawyerable_op:eq5}, \eqref{thm:calderon_thm_for_sawyerable_op:eq3}, \eqref{thm:calderon_thm_for_sawyerable_op:eq11}, a change of variables, \eqref{prel:two_stars_on_Lorentz_spaces}, and \eqref{thm:calderon_thm_for_sawyerable_op:eq2}, we have
\begin{align*}
   \|T f\|_{L^p(\cS, \nu)}^p &= \|(Tf)_\nu^*\|_{L^p(0, \nu(\cS))}^p = \|(\tilde T f)^*\|_{L^p(0, \nu(\cS))}^p = \|\tilde T f\|_{L^p(0, \nu(\cS))}^p \\
	&\leq \|(Sf)^{**}(t^{\ptwo })t^{\frac{\ptwo  - 1}{p}}\|_{L^p(0, \nu(\cS))}^p = \int_0^{\nu(\cS)}(Sf)^{**}(t^{\ptwo })^pt^{\ptwo  - 1} \dt \\
    &\leq \frac1{\ptwo } \|(Sf)^{**}\|_{L^p(0, \infty)}^p \lesssim \|Sf\|_{L^p(0, \infty)}^p = \|Sf\|_{L^p(0, \mu(\RR))}^p \\
    &\leq \|f\|_{L^p(\RR, \mu)}^p
\end{align*}
for every $f\in L^p(\RR, \mu)$. Moreover (using also~\eqref{def:sawyerable_op:def_of_p}), we have 
\begin{align*}
    \|T f\|_{L^\infty(\cS, \nu)} &= \|(T f)^*_\nu\|_{L^\infty(0, \nu(\cS))} = \|\tilde T f\|_{L^\infty(0, \nu(\cS))} \\
    &\leq \sup_{t\in(0, \infty)} (Sf)^{**}(t^{\ptwo })t^{\frac{\ptwo  - 1}{p}} = \sup_{t\in(0, \infty)} (Sf)^{**}(t^{\ptwo })t^{\frac{\ptwo }{\pone }} = \|(Sf)^{**}\|_{L^{\pone , \infty}(0, \infty)} \\
		&\lesssim \|Sf\|_{L^{\pone , \infty}(0, \infty)}  = \|Sf\|_{L^{\pone , \infty}(0, \mu(\RR))} \leq \|f\|_{L^{\pone , \infty}(\RR, \mu)}
\end{align*}
for every $f\in L^{\pone , \infty}(\RR, \mu)$. Hence, $T$ is $(p,\pone)$-sawyearable. At last, we are in a position to prove that (i)~implies~(iii). If (i) is assumed, it follows that $T$ is bounded from $X(\RR,\mu)$ to $Y^{\langle p \rangle}(\cS, \nu)$. Therefore, there is a constant $C$ such that
\begin{equation}\label{thm:calderon_thm_for_sawyerable_op:eq7}
\|Tf\|_{Y^{\langle p \rangle}(\cS, \nu)} \leq C \|f\|_{X(\RR,\mu)} \quad \text{for every $f\in X(\RR, \mu)$}.
\end{equation}
Consequently, using a change of variables, \eqref{thm:calderon_thm_for_sawyerable_op:eq12}, \eqref{thm:calderon_thm_for_sawyerable_op:eq6}, the Hardy--Littlewood inequality~\eqref{prel:HL_ineq} together with \eqref{thm:calderon_thm_for_sawyerable_op:eq4}, \eqref{thm:calderon_thm_for_sawyerable_op:eq7}, and \eqref{thm:calderon_thm_for_sawyerable_op:eq8}, we obtain
\begin{align*}
    \|R_{p,\pone}g\|_{\bar Y(0, \nu(\cS))} &=  \Big\| \Big(\frac1{t}\int_0^{t^{\ptwo }} g^*(s)^p \ds \Big)^\frac1{p} \Big\|_{\bar Y(0, \nu(\cS))} \\
    &\approx \Big\| \Big(\frac1{t}\int_0^t g^*(s^{\ptwo })^p s^{\ptwo  - 1} \ds \Big)^\frac1{p} \Big\|_{\bar Y(0, \nu(\cS))} \\
    &= \Big\| \Big(\frac1{t}\int_0^t \Big( S|h|(s^{\ptwo }) s^{\frac{\ptwo  - 1}{p}} \Big)^p \ds \Big)^\frac1{p} \Big\|_{\bar Y(0, \nu(\cS))} \\
    &\leq \Big\| \Big(\frac1{t}\int_0^t \tilde T |h| (s)^p  \ds \Big)^\frac1{p} \Big\|_{\bar Y(0, \nu(\cS))} \\
    &\leq \|((T|h|)^p)_\nu^{**})^\frac1{p}\|_{\bar Y(0, \nu(\cS))} = \|T|h|\|_{Y^{\langle p \rangle}(\cS, \nu)} \\
		&\leq C \|h\|_{X(\RR, \mu)} = C \|h^*_\mu\|_{\bar X(0, \mu(\RR))} = C \|g\|_{\bar X(0, \mu(\RR))}.
\end{align*}
Hence, the operator $R_{p,\pone}$ is bounded from $\bar X(0, \mu(\RR))$ to $\bar Y(0, \nu(\cS))$. In other words, we have shown that (i)~implies~(iii), which finishes the proof.
\end{proof}

Before we give a proof of Theorem~\ref{T:general-principle}, we need to make a simple observation.

\begin{lemma}\label{L:transfer-of-sawyerability}
Let $(\RR,\mu)$ and $(\cS, \nu)$ be nonatomic $\sigma$-finite measure spaces. Let $p,\pone$ satisfy~\eqref{E:pq}, and let $\alpha\in(\frac{1}{p},\infty)$. Assume that $T$ is an order preserving $(p,\pone)$-sawyerable operator. Then the operator $T_{\alpha}$, defined as
    \begin{equation*}
        T_{\alpha}f=\left[T\left(|f|^{\alpha}\right)\right]^{\frac{1}{\alpha}},
    \end{equation*}
    is $(\alpha p,\alpha \pone)$-sawyerable.
\end{lemma}

\begin{proof}
We first observe that the assumption $T$ is quasi-linear and order preserving easily implies that $T_\alpha$ is quasi-linear and order preserving. Fix $t\in(0, \nu(\cS))$. In view of Proposition~\ref{prop:joint_sawyerable_vs_sep_sawyerable}, the assertion immediately follows from observing that the value of $r$ which corresponds to the pair $(p,q)$ in~\eqref{def:sawyerable_op:def_of_p} is unchanged for any pair $(\alpha p,\alpha q)$, $\alpha \in (\frac{1}{p },\infty)$, whence one has
    \begin{align*}
            \left(|T_{\alpha}f|^{\alpha p}\right)_\nu^{**}(t)^{\frac{1}{\alpha p}}
            &=
            \left(|T\left(|f|^{\alpha}\right)|^{p}\right)_\nu^{**}(t)^{\frac{1}{\alpha p}}
                \nonumber\\
            &\lesssim
            \Big( R_{p,\pone}(|f|^{\alpha})_\mu^*(t) + \sup_{s\in[t^{\ptwo }, \infty)} s^\frac1{\pone } (|f|^{\alpha})_\mu^*(s)  \Big)^{\frac{1}{\alpha}}
                \nonumber\\
            &\approx
            R_{\alpha p,\alpha \pone}f_\mu^*(t) 
            + \sup_{s\in[t^{\ptwo }, \infty)} s^\frac1{\alpha \pone } f_\mu^*(s),
    \end{align*}
		in which the multiplicative constants depend only on $\alpha$.
\end{proof}

\begin{proof}[Proof of Theorem~\ref{T:general-principle}]
    We first note that, by Lemma~\ref{L:transfer-of-sawyerability}, the operator $G_{\frac{1}{\lambda}}$ is $(p,\pone)$-sawyerable. Furthermore, it is easy to see that the assumptions that $G$ is quasi-linear and order preserving imply that $G_{\frac{1}{\lambda}}$ is also quasi-linear. Consequently, since $R_{p,\pone}\colon \bar X(0, \mu(\RR))\to \bar Y(0, \nu(\cS))$, Theorem~\ref{thm:calderon_thm_for_sawyerable_op} implies that
\begin{equation}\label{E:general-transferred-sawyerability}
    \|G_{\frac1{\lambda}}f\|_{Y^{\langle p\rangle}(\cS,\nu)} \lesssim \|f\|_{X(\RR, \mu)} \quad \text{for every $f\in X(\RR, \mu)$}.
\end{equation}
Therefore, using \eqref{E:pointwise-general}, the definition of $G_{\frac1{\lambda}}$, and \eqref{E:general-transferred-sawyerability}, we arrive at
    \begin{equation}\label{E:general-riesz-chain}
        \|B\chi_{E}\|_{Y^{\langle p\rangle}(\cS,\nu)}
        \lesssim
         \|(G\chi_{E})^{\lambda}\|_{Y^{\langle p\rangle}(\cS,\nu)} =
         \|G_{\frac{1}{\lambda}}\chi_{E}\|_{Y^{\langle p\rangle}(\cS,\nu)}\lesssim
         \|\chi_{E}\|_{X(\RR, \mu)}
\end{equation}
for every $E\subseteq\RR$ of finite measure.
Using the summability property of $\Lambda_X(\RR,\mu)$ (cf.~\cite[Theorem 3.13 on p.~195]{SteinWeiss}), we get that \eqref{E:general-riesz-chain} is in fact valid for every simple function (i.e., a linear combination of characteristic functions of sets of finite measure) $f$ on $(\RR, \mu)$. To verify this argument, note that it is enough to consider nonnegative simple functions. Writing $f=\sum_{j=1}^N \alpha_j \chi_{E_j}$, where $\alpha_j > 0$, $j = 1, \dots, N$, $E_N\subseteq \cdots \subseteq E_1$, and using \eqref{E:general-riesz-chain} together with \eqref{T:general-principle:fundX_vanishes}, we obtain
\begin{align*}
    \|Bf\|_{Y^{\langle p\rangle}(\cS,\nu)} &\leq \sum_{j = 1}^N \alpha_j \|B\chi_{E_j}\|_{Y^{\langle p\rangle}(\cS,\nu)} \lesssim \sum_{j = 1}^N \alpha_j \varphi_X(\mu(E_j)) \\
    &= \sum_{j = 1}^N \alpha_j \int_0^{\mu(E_j)} \varphi_X'(t) \dt = \int_0^\infty f^*_\mu(t) \varphi_X'(t) \dt \\
    &\leq \; \|f\|_{\Lambda_X(\RR, \mu)}.
\end{align*}
Finally, since $B$ is linear and simple functions are dense in $\Lambda_X(\RR, \mu)$ thanks to \eqref{T:general-principle:fundX_vanishes}, the operator $B$ can be uniquely extended to a bounded linear operator from $\Lambda_X(\RR, \mu)$ to $Y^{\langle p\rangle}(\cS,\nu)$.
\end{proof}

\begin{remark}
The operator $B$ is assumed to be linear in Theorem~\ref{T:general-principle}. Another possibility is to assume that $B$ is a nonnegative sublinear operator (see~\cite[p.~230]{BS}) defined on all simple functions.
\end{remark}

We will precede the rather involved proof of Theorem~\ref{thm:R_boundedness_lorentz_infty} with that of Theorem~\ref{mainresult-prime}, which is just a corollary of the more general Theorem~\ref{thm:most_general_KK}.

\begin{proof}[Proof of Theorem~\ref{mainresult-prime}]
    By Proposition~\ref{prop:endpoint_estimates_for_R}, we get
    \begin{equation*}
        R_{p,\frac{d}{\alpha}}\colon 
        L^{p}(0, \mu(\rn)) 
        \to L^{p,\infty}(0, \nu(\rn)).
    \end{equation*}
    It thus follows from Theorem~\ref{thm:most_general_KK} with $Q_0=[0,1)^n$, $X(\rn,\mu) = L^{p}(\rn, \mu)$ and 
    $Y(\rn,\nu)= L^{p,\infty}(\rn, \nu)$ (note that the assumption~\eqref{T:general-principle:fundX_vanishes} is satisfied) that
    \begin{equation*}
        I_\alpha^\mu\colon \Lambda_{L^{p}}(\rn,\mu) \to (L^{p,\infty})^{\langle p \rangle}(\rn, \nu).
    \end{equation*}
    A straightforward application of~\eqref{prel:Lorentz_spaces_Lambda_endpoint} and \eqref{ex:sawyerable_T_boundedness_lorentz_infty:power} now yields~\eqref{kristensen-korobkov-two-measures}.
\end{proof}

We will round off this section with the proof of Theorem~\ref{thm:R_boundedness_lorentz_infty}. While being somewhat lengthy, we believe that it provides a valuable insight into the subject, and at the same time reveals interesting connections with fine properties of certain scales of function spaces and the interpolation theory (see Remark~\ref{rem:R_boundedness_lorentz_infty_remarks} for more detail).
\begin{proof}[Proof of Theorem~\ref{thm:R_boundedness_lorentz_infty}]
For the sake of brevity, we will write $\|\cdot\|_{r_j,s_j}$ and $\|\cdot\|_{s_j}$, $j=1,2$, instead of $\|\cdot\|_{L^{r_j,s_j}(0, \infty)}$ and $\|\cdot\|_{L^{s_j}(0,\infty)}$, respectively. The inequality~\eqref{E:R-on-lorentz} reads as
    \begin{equation}\label{E:R-on-lorentz-norms}
        \left\|\left(
        t^{\ptwo -1}(|g|^p)^{**}(t^{\ptwo })
        \right)^{\frac{1}{p}}
        \right\|_{r_2,s_2}
        \lesssim
        \|g\|_{r_1,s_1} \quad \text{for every $g\in\M(0, \infty)$}.
    \end{equation}
    On substituting $h=|g|^p$ and using the definition of Lorentz (quasi)norm on the right-hand side, we find that~\eqref{E:R-on-lorentz-norms} holds if and only if
    \begin{equation}\label{E:norm-formulation-1}
        \left\|\left(
        t^{\ptwo -1}h^{**}(t^{\ptwo })
        \right)^{\frac{1}{p}}
        \right\|_{r_2,s_2}
        \lesssim
        \left\|t^{\frac1{r_1}-\frac1{s_1}}h^*(t)^{\frac{1}{p}}\right\|_{s_1} \quad \text{for every $h\in\M(0, \infty)$.}
    \end{equation}
	Now, assume for the moment that we know that
  \begin{align}
      \left\|\left(
        t^{\ptwo -1}h^{**}(t^{\ptwo })
        \right)^{\frac{1}{p}}
        \right\|_{r_2,s_2} &\approx \left\|
        \sup_{t\leq \tau < \infty}\big(\tau^{\ptwo -1}h^{**}(\tau^{\ptwo })
        \big)^{\frac{1}{p}}
        \right\|_{r_2,s_2} \label{E:norm-formulation-equivalent_with_sup}\\
        \intertext{and}
        \left\|
        t^{\frac{1}{\ptwo r_2}-\frac{1}{s_2}}
        \sup_{t\le y<\infty}
        y^{\frac{1}{\pone }}h^{**}(y)
        ^{\frac{1}{p}}
        \right\|_{s_2} &\approx \left\|
        t^{\frac{1}{\ptwo r_2} + \frac{1}{\pone } -\frac{1}{s_2}} h^{**}(t)
        ^{\frac{1}{p}}
        \right\|_{s_2} \label{E:norm-formulation-getting_rid_of_sup}
  \end{align}
  for every $h\in\M(0, \infty)$. Using the definition of the Lorentz (quasi)norm and the change of variables $y=\tau^{\ptwo }$ inside the supremum, we see that
  \begin{align*}
      \left\|
        \sup_{t\leq \tau < \infty}\big(\tau^{\ptwo -1}h^{**}(\tau^{\ptwo })
        \big)^{\frac{1}{p}}
        \right\|_{r_2,s_2} &= \left\|
        t^{\frac{1}{r_2}-\frac{1}{s_2}}
        \sup_{t\le\tau<\infty}
        \left(
        \tau^{\ptwo -1}h^{**}(\tau^{\ptwo })
        \right)^{\frac{1}{p}}
        \right\|_{s_2} \\
        &=
        \left\|
        t^{\frac{1}{r_2}-\frac{1}{s_2}}
        \sup_{t^{\ptwo }\le y<\infty}
        \big(
        y^{1-\frac{1}{\ptwo }}h^{**}(y)
        \big)^{\frac{1}{p}}
        \right\|_{s_2}
  \end{align*}
for every $h\in\M(0, \infty)$. Moreover, on calling~\eqref{def:sawyerable_op:def_of_p} into play and using the change of variables $\tau=t^{\ptwo }$ (and renaming $\tau$ to $t$ again), we have
    \begin{align*}
        \left\|
        t^{\frac{1}{r_2}-\frac{1}{s_2}}
        \sup_{t^{\ptwo }\le y<\infty}
        \big(
        y^{1-\frac{1}{\ptwo }}h^{**}(y)
        \big)^{\frac{1}{p}}
        \right\|_{s_2} &= \left\|
        t^{\frac{1}{r_2}-\frac{1}{s_2}}
        \sup_{t^{\ptwo }\le y<\infty}
        y^{\frac{1}{\pone }}h^{**}(y)
        ^{\frac{1}{p}}
        \right\|_{s_2} \\
        &\approx\left\|
        t^{\frac{1}{\ptwo r_2}-\frac{1}{s_2}}
        \sup_{t\le y<\infty}
        y^{\frac{1}{\pone }}h^{**}(y)
        ^{\frac{1}{p}}
        \right\|_{s_2}
    \end{align*}
		for every $h\in\M(0, \infty)$. Therefore, putting all these things together and assuming that both \eqref{E:norm-formulation-equivalent_with_sup} and \eqref{E:norm-formulation-getting_rid_of_sup} are true, we have shown that \eqref{E:norm-formulation-1} holds if and only if
  \begin{equation}
  \label{E:norm-formulation-2} 
      \left\|
        t^{\frac{1}{\ptwo r_2} + \frac{1}{\pone } -\frac{1}{s_2}} h^{**}(t)
        ^{\frac{1}{p}}
        \right\|_{s_2} \lesssim \left\|t^{\frac1{r_1}-\frac1{s_1}}h^*(t)^{\frac{1}{p}}\right\|_{s_1} \quad \text{for every $h\in\M(0, \infty)$.}
  \end{equation}
 Since the validity of \eqref{E:norm-formulation-1} is equivalent to that of \eqref{E:R-on-lorentz-norms}, and so also to the desired boundedness~\eqref{E:R-on-lorentz}, we arrive at
 \begin{equation}\label{E:norm-formulation-3}
     \text{$R_{p,\pone}\colon 
            L^{r_1,s_1}(0,\infty)\to L^{r_2,s_2}(0, \infty)$ is bounded if and only if \eqref{E:norm-formulation-2} is true},
 \end{equation}
 provided that \eqref{E:norm-formulation-equivalent_with_sup} and \eqref{E:norm-formulation-getting_rid_of_sup} are valid.

 We now prove \eqref{E:norm-formulation-equivalent_with_sup} and \eqref{E:norm-formulation-getting_rid_of_sup}, starting with the former. On the one hand, using the lattice property of the Lorentz (quasi)norm,~\eqref{E:lattice},
    and observing that the function
    \begin{equation*}
        t\mapsto \sup_{t\le\tau<\infty}
        \left(
        \tau^{\ptwo -1}h^{**}(\tau^{\ptwo })
        \right)^{\frac{1}{p}}
    \end{equation*}
    is nonincreasing on $(0,\infty)$ and clearly majorizes the function
    \begin{equation*}
        t\mapsto \left(
        t^{\ptwo -1}h^{**}(t^{\ptwo })
        \right)^{\frac{1}{p}},
    \end{equation*}
    we see that
		\begin{equation}\label{E:norm-formulation-equivalent_with_sup_eq1}
			\left\|
        \left(
        t^{\ptwo -1}h^{**}(t^{\ptwo })
        \right)^{\frac{1}{p}}
        \right\|_{r_2,s_2} \leq \left\|
        \sup_{t\le\tau<\infty}
        \left(
        \tau^{\ptwo -1}h^{**}(\tau^{\ptwo })
        \right)^{\frac{1}{p}}
        \right\|_{r_2, s_2}
		\end{equation}
		for every $h\in\M(0, \infty)$. On the other hand, by \cite[Lemma~3.1(ii)]{GP-Indiana:09} (with $\beta = 0$ and $\alpha = (\ptwo -1)/p$ in their notation), we have
  \begin{equation*}
      \int_0^t \sup_{s\le\tau<\infty}
        \left(
        \tau^{\ptwo -1}h^{**}(\tau^{\ptwo })
        \right)^{\frac{1}{p}} \ds \lesssim \int_0^t
        \left(\big(\tau^{\ptwo -1}h^{**}(\tau^{\ptwo }) \big)^{\frac{1}{p}} \right)^*(s) \ds \quad \text{for every $t\in(0,\infty)$}
  \end{equation*}
  and every $h\in\M(0, \infty)$. Consequently,
  \begin{equation}\label{E:norm-formulation-equivalent_with_sup_eq2}
      \left\|
        \sup_{t\le\tau<\infty}
        \left(
        \tau^{\ptwo -1}h^{**}(\tau^{\ptwo })
        \right)^{\frac{1}{p}}\right\|_{r_2,s_2} \lesssim \left\|
        \big(t^{\ptwo -1}h^{**}(t^{\ptwo }) \big)^{\frac{1}{p}}\right\|_{r_2,s_2}
  \end{equation}
  for every $h\in\M(0, \infty)$ by virtue of the Hardy--Littlewood--P\'olya principle (see~\cite[Chapter~2, Theorem~4.6]{BS}). Hence, combining \eqref{E:norm-formulation-equivalent_with_sup_eq1} and \eqref{E:norm-formulation-equivalent_with_sup_eq2}, we obtain \eqref{E:norm-formulation-equivalent_with_sup}. 
  
  Next, we turn our attention to \eqref{E:norm-formulation-getting_rid_of_sup}. We will distinguish between $s_2<\infty$ and $s_2=\infty$, the latter case being considerably simpler. When $s_2 = \infty$, we simply interchange the suprema to obtain
  \begin{align*}
      \left\|
        t^{\frac{1}{\ptwo r_2}-\frac{1}{s_2}}
        \sup_{t\le y<\infty}
        y^{\frac{1}{\pone }}h^{**}(y)
        ^{\frac{1}{p}}
        \right\|_{s_2} &= \sup_{t\in(0,\infty)}t^{\frac{1}{\ptwo r_2}}
        \sup_{t\le y<\infty}
        y^{\frac{1}{\pone }}h^{**}(y)
        ^{\frac{1}{p}} \\
        &=\sup_{y\in(0,\infty)}y^{\frac{1}{\pone }}h^{**}(y)
        ^{\frac{1}{p}} \sup_{0<t\leq y}t^{\frac{1}{\ptwo r_2}} \\
        &=\sup_{y\in(0,\infty)}y^{\frac{1}{\ptwo r_2} + \frac{1}{\pone }}h^{**}(y)
        ^{\frac{1}{p}} \\
        &=\|t^{\frac{1}{\ptwo r_2} + \frac{1}{\pone } - \frac1{s_2}}h^{**}(t)
        ^{\frac{1}{p}}\|_{s_2}
  \end{align*}
  for every $h\in\M(0, \infty)$. Hence \eqref{E:norm-formulation-getting_rid_of_sup} is true when $s_2 = \infty$ (in fact, with equality). 
  
  Assume now that $s_2<\infty$. We clearly have
  \begin{equation}\label{E:norm-formulation-getting_rid_of_sup_eq1}
      \left\|
        t^{\frac{1}{\ptwo r_2} + \frac{1}{\pone } -\frac{1}{s_2}} h^{**}(t)
        ^{\frac{1}{p}}
        \right\|_{s_2}\leq \left\|
        t^{\frac{1}{\ptwo r_2}-\frac{1}{s_2}}
        \sup_{t\le y<\infty}
        y^{\frac{1}{\pone }}h^{**}(y)
        ^{\frac{1}{p}}
        \right\|_{s_2}
  \end{equation}
   for every $h\in\M(0, \infty)$, and so we only need to prove the converse inequality. The desired inequality follows from~\cite[Theorem~3.2(i)]{Gog:06}. We sketch the way in which their theorem is used for the reader's convenience. Fix $h\in\M(0,\infty)$, and denote
    \begin{equation*}
        \varphi(y)=
        h^{**}(y)
        ^{\frac1{p}}
        \quad\text{for $y\in(0,\infty)$}
    \end{equation*}
    and
    \begin{equation*}
       u(y)=
        y^{\frac1{\pone }},
        \quad
       v(y)=
        y^{\frac{s_2}{\ptwo r_2} + \frac{s_2}{\pone }
        -1},
        \quad
       w(y)=
        y^{\frac{s_2}{\ptwo r_2}-1}        
        \quad\text{for $y\in(0,\infty)$.}
    \end{equation*}
    An easy calculation shows that
    \begin{equation*}
        \left(
        \int_{0}^{x}
        \left[\sup_{t\le\tau\le x}u(\tau)\right]
        ^{s_2}w(t)\dt\right)^{\frac{1}{s_2}}
        \lesssim
        \left(
        \int_{0}^{x}
        v(t)\dt\right)^{\frac{1}{s_2}}
        \quad\text{for every $x\in(0,\infty)$.}
    \end{equation*}
    Therefore, applying~\cite[Theorem~3.2(i)]{Gog:06} to (in their notation) $p=q=s_2$ and $u,v,w$ as above, and noticing that $\varphi$ is obviously nonincreasing on $(0,\infty)$, we obtain
    \begin{equation}\label{E:norm-formulation-getting_rid_of_sup_eq2}
\left\|
        t^{\frac{1}{\ptwo r_2}-\frac{1}{s_2}}
        \sup_{t\le y<\infty}
        y^{\frac{1}{\pone }}h^{**}(y)
        ^{\frac{1}{p}}
        \right\|_{s_2} \lesssim  \left\|
        t^{\frac{1}{\ptwo r_2} + \frac{1}{\pone } -\frac{1}{s_2}} h^{**}(t)
        ^{\frac{1}{p}}
        \right\|_{s_2},
    \end{equation}
    in which the multiplicative constant does not depend on $h$. Finally, combining \eqref{E:norm-formulation-getting_rid_of_sup_eq1} and \eqref{E:norm-formulation-getting_rid_of_sup_eq2}, we obtain \eqref{E:norm-formulation-getting_rid_of_sup} even when $s_2<\infty$.

    The above analysis shows that the inequality~\eqref{E:R-on-lorentz} holds if and only if~\eqref{E:norm-formulation-2} is satisfied. It thus only remains to verify that the validity of~\eqref{E:norm-formulation-2} is equivalent to that of one of the conditions~\eqref{E:R-lorentz-parameters-i}--\eqref{E:R-lorentz-parameters-iii}. We shall split the proof of this fact into four parts in dependence on finiteness or non-finiteness of the parameters $s_1$ and $s_2$, since the techniques are different for each of these cases. We will need some knowledge from the theory of weighted inequalities on the cone of monotone functions.

    (a) Assume that $s_1<\infty$ and $s_2<\infty$. Then,~\eqref{E:norm-formulation-2} (after raising it to $p$) reads as
    \begin{equation}\label{E:R-on-lorentz-3}
        \left(\int_{0}^{\infty}
         h^{**}(t)^{\frac{s_2}{p}}
        t^{\frac{s_2}{\pone }+\frac{s_2}{\ptwo r_2}-1}\dt
        \right)^{\frac{p}{s_2}}
            \lesssim
        \left(\int_{0}^{\infty}
        h^*(t)^{\frac{s_1}{p}}t^{\frac{s_1}{r_1}-1}\dt
        \right)^{\frac{p}{s_1}}
    \end{equation}
		for every $h\in\M(0, \infty)$. A simple analysis of \cite[Theorem~10.3.12, (iii)--(vi)]{Pic:13} reveals that~\eqref{E:R-on-lorentz-3} cannot hold unless $s_1\le s_2$. To this end, one can easily observe that the condition~\cite[(10.3.13)]{Pic:13} cannot be satisfied, because it requires an integral of a power function over $(0, \infty)$ to be finite. So, we have to have
    \begin{equation}\label{E:s-2-and-s-2}
        s_1\le s_2
    \end{equation}
		for \eqref{E:R-on-lorentz-3} to possibly hold. A characterization of the validity of \eqref{E:R-on-lorentz-3} under the restriction \eqref{E:s-2-and-s-2} is provided by~\cite[Theorem~10.3.12, (i)--(ii)]{Pic:13}. It reads as: either $s_1>p$,
    \begin{equation}\label{E:lorentz-condition-1}
        \sup_{t\in(0,\infty)}
        \frac{
        \left(
        \int_{0}^{t}\tau^{\frac{s_2}{\pone }+\frac{s_2}{\ptwo r_2}-1}\dtau
        \right)^{\frac{p}{s_2}}
        }
        {
        \left(
        \int_{0}^{t}\tau^{\frac{s_1}{r_1}-1}\dtau\right)^{\frac{p}{s_1}}}
        <\infty,
    \end{equation}
    and
    \begin{equation*}
        \sup_{t\in(0,\infty)}
        \left(
        \int_{t}^{\infty}
        \tau^{\frac{s_2}{\pone }+\frac{s_2}{\ptwo r_2}-\frac{s_2}{p}-1}\dtau
        \right)^{\frac{p}{s_2}}
        \left(
        \int_{0}^{t}
        \tau^{\frac{s_1}{r_1}-1+\frac{s_1}{s_1-p}-\frac{s_1}{r_1}\frac{s_1}{s_1-p}}
        \dtau\right)^{\frac{s_1-p}{s_1}}
        <\infty,
    \end{equation*}
    or $s_1\le p$,~\eqref{E:lorentz-condition-1} is satisfied, and
    \begin{equation*}
        \sup_{t\in(0,\infty)}
        t\left(
        \int_{t}^{\infty}
        \tau^{\frac{s_2}{\pone }+\frac{s_2}{\ptwo r_2}-\frac{s_2}{p}-1}\dtau
        \right)^{\frac{p}{s_2}}
        \left(
        \int_{0}^{t}
        \tau^{\frac{s_1}{r_1}-1}
        \dtau\right)^{-\frac{p}{s_1}}
        <\infty.
    \end{equation*}
    Calculation shows that all these conditions are equivalent to
    \begin{align}
        &r_1>p,\quad                       r_2>p,\quad\text{and}
        \quad\frac{1}{\pone }
        +\frac{1}{\ptwo r_2}
        =\frac{1}{r_1}.
        \label{E:R-interresult-i}
    \end{align}
    We next observe that $r_2>p$ is in fact superfluous in~\eqref{E:R-interresult-i} as it follows from the other two relations. 
    Indeed,~the combination of $r_1>p$ with 
    $\frac{1}{\pone }
    +\frac{1}{\ptwo r_2}
    =\frac{1}{r_1}$
    directly enforces
   \begin{equation}\label{E:lorentz-parameters-inequality}
        \frac{1}{\pone }+\frac{1}{\ptwo r_2}<\frac{1}{p}.
    \end{equation}
    By~\eqref{def:sawyerable_op:def_of_p},~\eqref{E:lorentz-parameters-inequality} is equivalent to 
    \begin{equation*}
        \frac{1}{\pone }+\frac{1}{\ptwo r_2}<\frac{\ptwo '}{\pone },
    \end{equation*}
    whence, using~\eqref{def:sawyerable_op:def_of_p} once again, one gets
    \begin{equation*}
       \frac{1}{r_2}<\frac{\ptwo (\ptwo '-1)}{\pone }=\frac{1}{p},
    \end{equation*}
    and the claim follows. So,~\eqref{E:R-interresult-i} is equivalent to
        \begin{align}
        &r_1>p\quad\text{and}
        \quad\frac{1}{\pone }
        +\frac{1}{\ptwo r_2}
        =\frac{1}{r_1}.
        \label{E:R-interresult-ii}
    \end{align}
    Since \eqref{E:R-interresult-ii} immediately implies that $r_1<\pone $, we conclude that, in the case~(a),~\eqref{E:norm-formulation-2}, hence~\eqref{E:norm-formulation-3}, holds if and only if~\eqref{E:R-lorentz-parameters-i} does.

    (b) Assume that $s_1<\infty$ and $s_2=\infty$. We claim that then~\eqref{E:norm-formulation-2} holds if and only if
    \begin{equation}
    \label{E:balance-condition}
       \frac{1}{\pone }
        +\frac{1}{\ptwo r_2}
        =\frac{1}{r_1}
    \end{equation}
    and 
    \begin{equation*}
       \begin{cases}
           \text{either} &r_1\ge p\text{ and $s_1\le p$}
                \\
           \text{or} &r_1> p\text{ and $s_1> p$.} 
       \end{cases}
    \end{equation*}
    To verify this claim, note that~\eqref{E:norm-formulation-2}, raised to $p$, turns into 
    \begin{equation}
    \label{E:norm-formulation-12}
        \sup_{y\in(0,\infty)}
        h^{**}(y)
        y^{\frac{p}{\pone }+\frac{p}{\ptwo r_2}}
        \lesssim
        \left(\int_{0}^{\infty}h^*(t)^{\frac{s_1}{p}}t^{\frac{s_1}{r_1}-1}\ds\right)^{\frac{p}{s_1}}
    \end{equation}
		for every $h\in\M(0, \infty)$. To characterize parameters for which~\eqref{E:norm-formulation-12} holds, we will exploit~\cite[Theorem~3.15]{Gog:13}, which (translated to our notation) states that this inequality holds if and only if either $s_1\le p$ and 
    \begin{equation*}
        \sup_{t\in(0,\infty)}\sup_{\tau\in(0,\infty)}\min\{t,\tau\}\tau^{\frac{p}{\pone }+\frac{p}{\ptwo r_2}-1}t^{-\frac{p}{r_1}}<\infty,
    \end{equation*}
    or $s_1> p$ and 
    \begin{equation}\label{E:GS-condition-ii}
        \sup_{t\in(0,\infty)}
        t^{\frac{p}{\pone }+\frac{p}{\ptwo r_2}-1}
        \left(\int_{0}^{t}
        \left(\int_{\tau}^{t}
        s^{-\frac{s_1}{r_1}}\ds
        \right)^{\frac{s_1}{s_1-p}}
       \tau^{\frac{s_1}{r_1}-1}\dtau
        \right)^{\frac{s_1-p}{s_1}}<\infty.
    \end{equation}
    Calculation shows that in the first case, that is, when $s_1\le p$, the necessary and sufficient condition for~\eqref{E:norm-formulation-12} is $p\le r_1$ 
    and~\eqref{E:balance-condition}. In the second case, that is, when $s_1>p$, the analysis is more complicated because of the kernel occurring in the condition. Notice that, owing to the standard techniques, the term 
    $\int_{\tau}^{t}
    s^{-\frac{s_1}{r_1}}\ds$ 
    can be equivalently replaced in~\eqref{E:GS-condition-ii} by
    \begin{equation*}
        \begin{cases}
            t^{1-\frac{s_1}{r_1}}
            &\text{if $s_1<r_1$}
                \\
           \log\frac{t}{\tau}
            &\text{if $s_1=r_1$}
                \\
           \tau^{1-\frac{s_1}{r_1}}
            &\text{if $s_1>r_1$}.
        \end{cases}
    \end{equation*}
    Some more computation shows that if $s_1<r_1$, then the desired inequality holds if and only if~\eqref{E:balance-condition} holds. In the second case, when $s_1=r_1$, the condition reads as
    \begin{equation*}
        \sup_{t\in(0,\infty)}
        t^{\frac{p}{\pone }+\frac{p}{\ptwo r_2}-1}
        \left(\int_{0}^{t}
        \left(\log\frac{t}{\tau}
        \right)^{\frac{s_1}{s_1-p}}
        \dtau
        \right)^{\frac{s_1-p}{s_1}}<\infty.
    \end{equation*}
    Homogeneizing the integral by changing variables $\tau=ty$, we obtain
    \begin{equation*}
    \begin{split}
        \sup_{t\in(0,\infty)}
        &t^{\frac{p}{\pone }+\frac{p}{\ptwo r_2}-1}
        \left(\int_{0}^{t}
        \left(\log\frac{t}{\tau}
        \right)^{\frac{s_1}{s_1-p}}
        \dtau
        \right)^{\frac{s_1-p}{s_1}}
            \\
        &=\sup_{t\in(0,\infty)}
        t^{\frac{p}{\pone }+\frac{p}{\ptwo r_2}-\frac{p}{s_1}}
        \left(\int_{0}^{1}
        \left(\log\frac{1}{y}
        \right)^{\frac{s_1}{s_1-p}}
        \dy
        \right)^{\frac{s_1-p}{s_1}}.
    \end{split}
    \end{equation*}
    Since the last integral is convergent, we see that this is, once again, equivalent to~\eqref{E:balance-condition}. Finally, if $s_1>r_1$, straightforward calculation shows that \eqref{E:GS-condition-ii} holds if and only if $p<r_1$ and~\eqref{E:balance-condition} holds. This establishes the assertion in the case (b).
    
       (c) Assume that $s_1=\infty$ and $s_2<\infty$. We claim that, in this case, the inequality~\eqref{E:norm-formulation-2} is impossible. Indeed, the choice $h(t) = h^*(t) = t^{-\frac{p}{r_1}}$, $t\in(0, \infty)$, makes the right-hand side of~\eqref{E:norm-formulation-2} finite, while making the left-hand side infinite regardless of the choice of the other parameters which have not been fixed. This establishes the claim.

    (d) Assume that $s_1=\infty$ and $s_2=\infty$. 
    Then,~\eqref{E:norm-formulation-2} reads as   
    \begin{equation}
    \label{E:norm-formulation-14}
        \sup_{t\in(0,\infty)}
        h^{**}(t)^{\frac{1}{p}}
        t^{\frac{1}{\pone }+\frac{1}{\ptwo r_2}}
        \lesssim
        \sup_{t\in(0,\infty)}
        t^{\frac{1}{r_1}}
        h^*(t)^{\frac{1}{p}}
    \end{equation}
		for every $h\in\M(0, \infty)$. We claim that~\eqref{E:norm-formulation-14} holds if and only if~\eqref{E:R-interresult-ii} does. Indeed, to verify the `only if' part,
    we assume that~\eqref{E:norm-formulation-14} holds and test it first on the single function $h(t)=h^*(t)=t^{-\frac{p}{r_1}}$, $t\in(0, \infty)$. This immediately shows that $r_1>p$ is necessary for~\eqref{E:norm-formulation-14} because otherwise $h^*$ is not integrable near zero, whence $h^{**}\equiv\infty$ on $(0,\infty)$, which makes the left-hand side of~\eqref{E:norm-formulation-14} infinite and  the right-hand side finite. As the next step, we test~\eqref{E:norm-formulation-14} on~$h=h_a=h^*_a=\chi_{(0,a)}$ for any fixed $a\in(0,\infty)$. We get 
    \begin{equation*}
        \sup_{t\in(0,\infty)}
        \left(\chi_{(0,a)}(t)+\frac{a}{t}\chi_{[a,\infty)}(t)\right)^{\frac{1}{p}}
        t^{\frac{1}{\pone }+\frac{1}{\ptwo r_2}}
        \lesssim
        a^{\frac{1}{r_1}},
    \end{equation*}
		for every $a\in(0, \infty)$, which, in turn, enforces
    \begin{equation}\label{E:norm-formulation-16a}
        a^{\frac{1}{\pone }+\frac{1}{\ptwo r_2}}
        \lesssim
        a^{\frac{1}{r_1}} \quad \text{for every $a\in(0, \infty)$}.
    \end{equation}
    A simple inspection shows that~\eqref{E:norm-formulation-16a} implies $\frac{1}{\pone }
        +\frac{1}{\ptwo r_2}
        =\frac{1}{r_1}$. Altogether, we see that \eqref{E:R-interresult-ii} is necessary for \eqref{E:norm-formulation-14}.
     
    Conversely, to establish the `if' part, let $h\in\M(0,\infty)$ be such that the expression on the right-hand side of~\eqref{E:norm-formulation-14} is finite, and denote
    \begin{equation}\label{E:M}
        M=\sup_{t\in(0,\infty)}
        t^{\frac{p}{r_1}}
        h^*(t) < \infty.
    \end{equation}
    Then one has
    \begin{equation*}
        h^*(t) \le Mt^{-\frac{p}{r_1}}
        \quad\text{for every $t\in(0,\infty)$.}
    \end{equation*}
    Integrating and using the fact that $p<r_1$, we get, 
    \begin{equation*}
        h^{**}(y) \le \frac{r_1}{r_1-p}My^{-\frac{p}{r_1}}
        \quad\text{for every $y\in(0,\infty)$.}
    \end{equation*}
    Consequently, owing to~\eqref{E:R-interresult-ii}, one has
     \begin{align}
        \sup_{y\in(0,\infty)}
        h^{**}(y)^{\frac{1}{p}}
        y^{\frac{1}{\pone }+\frac{1}{\ptwo r_2}}
        &\leq \Big( \frac{r_1}{r_1-p} M \Big)^\frac1{p} \sup_{y\in(0,\infty)}
        y^{\frac{1}{\pone }+\frac{1}{r_2\ptwo }-\frac{1}{r_1}} \nonumber\\
        &= \Big( \frac{r_1}{r_1-p} M \Big)^\frac1{p}.\label{E:M-up}
    \end{align}
    Hence, \eqref{E:norm-formulation-14} follows from the combination of~\eqref{E:M} and~\eqref{E:M-up}. This establishes the assertion in the case (d) and completes the 
    proof of the theorem.
\end{proof}

\begin{remark}\label{rem:R_boundedness_lorentz_infty_remarks}
The boundedness of $R_{p,\pone}\colon L^{r_1,s_1}(0,\infty)\to L^{r_2,s_2}(0, \infty)$ with the parameters satisfying \eqref{E:R-lorentz-parameters-i} can alternatively be derived from combining Proposition~\ref{prop:endpoint_estimates_for_R} with the Marcinkiewicz interpolation theorem (in its version for Lorentz spaces due to A.\,P.\ Calder\'on, e.g., see~\cite[Chapter~4, Theorem~4.13]{BS}). However, the benefit of the different proof given above is twofold (apart from being self-contained, to some extent).  First, it additionally shows the necessity of the restriction $s_1\leq s_2$. Second, it suggests a way in which one could obtain boundedness of $R_{p,\pone}$ between more general function spaces\textemdash in particular, between function spaces that are instances of the so-called Lorentz Lambda spaces introduced in~\cite{L:51} (see also~\cite[Chapter~10]{Pic:13} for more information).
\end{remark}

\section*{Acknowledgments}
We wish to thank the referees for carefully checking our manuscript and for their valuable comments.
Z.~Mihula and L.~Pick are supported by grant no.~23-04720S of the Czech Science Foundation.  
D. Spector is supported by the National Science and Technology Council of Taiwan under research grant numbers 110-2115-M-003-020-MY3/113-2115-M-003-017-MY3 and the Taiwan Ministry of Education under the Yushan Fellow Program.

\end{document}